\documentclass[11pt]{amsart}
\usepackage[margin=1.5in]{geometry}

\usepackage{pinlabel}

\usepackage{adjustbox}
\usepackage{pinlabel}
       
\usepackage{adjustbox}

\usepackage[all,cmtip]{xy}
\usepackage{amsmath,amssymb,amsthm, amscd, xypic}
\usepackage{mathtools}
\usepackage[english]{babel}
\usepackage{tikz-cd, enumerate}
\usetikzlibrary{decorations.pathmorphing}
\tikzcdset{scale cd/.style={every label/.append style={scale=#1},
    cells={nodes={scale=#1}}}}
\usepackage{verbatim}
\usepackage{xfrac} 

\usepackage[normalem]{ulem}

\usepackage{todonotes}

\newtheorem{introthm}{Theorem}

\usepackage{xcolor}
\definecolor{darkgreen}{rgb}{0,0.50,0} 
\definecolor{darkred}{rgb}{0.55,0,0}
\definecolor{darkblue}{rgb}{0,0,0.6}

\usepackage[pdfborder=0,pagebackref,colorlinks,citecolor=darkgreen,linkcolor=darkred,urlcolor=darkblue]{hyperref}
\addto\extrasenglish{
    
}
\def\makeautorefname#1#2{\expandafter\def\csname#1autorefname\endcsname{#2}}
\makeautorefname{equation}{Equation}%
\makeautorefname{footnote}{footnote}%
\makeautorefname{item}{item}%
\makeautorefname{figure}{Figure}%
\makeautorefname{table}{Table}%
\makeautorefname{part}{Part}%
\makeautorefname{appendix}{Appendix}%
\makeautorefname{chapter}{Chapter}%
\makeautorefname{section}{Section}%
\makeautorefname{subsection}{Section}%
\makeautorefname{subsubsection}{Section}%
\makeautorefname{paragraph}{Paragraph}%
\makeautorefname{subparagraph}{Paragraph}%
\makeautorefname{theorem}{Theorem}%
\makeautorefname{thm}{Theorem}%
\makeautorefname{introthm}{Theorem}%
\makeautorefname{addm}{Addendum}%
\makeautorefname{mainthm}{Main theorem}%
\makeautorefname{corollary}{Corollary}%
\makeautorefname{cor}{Corollary}%
\makeautorefname{lemma}{Lemma}%
\makeautorefname{lem}{Lemma}%
\makeautorefname{lemmasection}{Lemma}%
\makeautorefname{sublemma}{Sublemma}%
\makeautorefname{sublem}{Sublemma}%
\makeautorefname{subl}{Sublemma}%
\makeautorefname{prop}{Proposition}%
\makeautorefname{property}{Property}
\makeautorefname{pro}{Property}
\makeautorefname{sch}{Scholium}%
\makeautorefname{step}{Step}%
\makeautorefname{conject}{Conjecture}%
\makeautorefname{conj}{Conjecture}%
\makeautorefname{questn}{Question}
\makeautorefname{quest}{Question}
\makeautorefname{qn}{Question}
\makeautorefname{definition}{Definition}%
\makeautorefname{defn}{Definition}%
\makeautorefname{defi}{Definition}%
\makeautorefname{def}{Definition}%
\makeautorefname{dfn}{Definition}%
\makeautorefname{definitionsection}{Definition}
\makeautorefname{notation}{Notation}
\makeautorefname{notn}{Notation}
\makeautorefname{rem}{Remark}%
\makeautorefname{rems}{Remarks}%
\makeautorefname{rmk}{Remark}%
\makeautorefname{rk}{Remark}%
\makeautorefname{remarks}{Remarks}%
\makeautorefname{rems}{Remarks}%
\makeautorefname{rmks}{Remarks}%
\makeautorefname{rks}{Remarks}%
\makeautorefname{example}{Example}%
\makeautorefname{examp}{Example}%
\makeautorefname{exmp}{Example}%
\makeautorefname{exam}{Example}%
\makeautorefname{exa}{Example}%
\makeautorefname{axiom}{Axiom}%
\makeautorefname{axi}{Axiom}%
\makeautorefname{ax}{Axiom}%
\makeautorefname{case}{Case}%
\makeautorefname{claim}{Claim}%
\makeautorefname{clm}{Claim}%
\makeautorefname{assumpt}{Assumption}%
\makeautorefname{asses}{Assumptions}%
\makeautorefname{conclusion}{Conclusion}%
\makeautorefname{concl}{Conclusion}%
\makeautorefname{conc}{Conclusion}%
\makeautorefname{cond}{Condition}%
\makeautorefname{const}{Construction}%
\makeautorefname{con}{Construction}%
\makeautorefname{criterion}{Criterion}%
\makeautorefname{criter}{Criterion}%
\makeautorefname{crit}{Criterion}%
\makeautorefname{exercise}{Exercise}%
\makeautorefname{exer}{Exercise}%
\makeautorefname{exe}{Exercise}%
\makeautorefname{problem}{Problem}%
\makeautorefname{problm}{Problem}%
\makeautorefname{prob}{Problem}%
\makeautorefname{prob}{Problem}%
\makeautorefname{soln}{Solution}%
\makeautorefname{sol}{Solution}%
\makeautorefname{sum}{Summary}%
\makeautorefname{oper}{Operation}%
\makeautorefname{obs}{Observation}%
\makeautorefname{ob}{Observation}%
\makeautorefname{conv}{Convention}%
\makeautorefname{cvn}{Convention}%
\makeautorefname{warn}{Warning}%
\makeautorefname{note}{Note}%
\makeautorefname{fact}{Fact}%
\makeautorefname{ouch0}{Counterexample}%
\theoremstyle{definition}
\newtheorem{definition}{Definition}[subsection]
\newtheorem{lemma}[definition]{Lemma}
\newtheorem{prop}[definition]{Proposition}
\newtheorem{remk}[definition]{Remark}

\theoremstyle{thm}
\newtheorem{thm}{Theorem}[section]

\newtheorem{cor}[thm]{Corollary}

\newtheorem{definitionsection}[thm]{Definition}
\newtheorem{lemmasection}[thm]{Lemma}

\newtheorem{remarksection}[thm]{Remark}

\newcommand{\im}{\mathrm{Im}}

\newcommand{\SK}{\mathrm{SK}}
\newcommand{\SKK}{\mathrm{SKK}}
\newcommand{\id}{\mathrm{id}}
\newcommand{\Cob}{\mathrm{Cob}}

\newcommand{\overbar}[1]{\mkern 3mu\overline{\mkern-3mu#1\mkern-4mu}\mkern 4mu}
\def\C{\mathcal{C}}
\def\d{\partial}

\newcommand{\mM}{\mathcal{M}}

\newcommand{\R}{\mathbb R}
\newcommand{\Z}{\mathbb Z}

\newcommand{\Map}{\textup{Triv}}

\newcommand{\Mnfldbd}{\textup{Mfd}^\partial}

\newcommand{\op}{\textup{op}}

\newcommand{\CEmb}{\text{SK-Emb}}

\usepackage{manfnt}
\newcommand{\SKbar}{\overline{\mathrm{SK}}}
\newcommand{\cube}{\text{\mancube}}

\newcommand{\mbar}{\overline{\mathrm{Mfd}}^\partial_n}

\makeatletter
\def\moverlay{\mathpalette\mov@rlay}
\def\mov@rlay#1#2{\leavevmode\vtop{%
   \baselineskip\z@skip \lineskiplimit-\maxdimen
   \ialign{\hfil$\m@th#1##$\hfil\cr#2\crcr}}}
   
\newcommand{\charfusion}[3][\mathord]{
    #1{\ifx#1\mathop\vphantom{#2}\fi
        \mathpalette\mov@rlay{#2\cr#3}
      }
    \ifx#1\mathop\expandafter\displaylimits\fi}
\makeatother   
\newcommand{\Tcup}{\charfusion[\mathop]{\bigsqcup}{\innerh}}
\newcommand{\innerh}{\mathchoice
{\scriptsize{\raisebox{3pt}{\textsc{T}}}}
{\tiny{\raisebox{2pt}{\textsc{T}}}}
{\tiny{\raisebox{2pt}{\textsc{T}}}}
{\tiny{\raisebox{2pt}{\textsc{T}}}}
}

\newcommand{\lout}{\bgroup\markoverwith
	{\textcolor{orange}{\rule[.5ex]{2pt}{0.4pt}}}\ULon}
\newcommand{\jout}{\bgroup\markoverwith
	{\textcolor{magenta}{\rule[.5ex]{2pt}{0.4pt}}}\ULon}
	
\newcommand{\cout}{\bgroup\markoverwith
	{\textcolor{darkgreen}{\rule[.6ex]{3pt}{0.6pt}}}\ULon}
	
	\newcommand{\mout}{\bgroup\markoverwith
	{\textcolor{blue}{\rule[.6ex]{3pt}{0.6pt}}}\ULon}

\newcommand{\renee}[1]{{\color{red}{#1}}}
\newcommand\rout{\bgroup\markoverwith
	{\textcolor{red}{\rule[.5ex]{2pt}{0.4pt}}}\ULon}

\usepackage{calc}

\newlength{\storeparskip}
\author[R. S. Hoekzema]{Renee  S. Hoekzema}
\address{Department of Mathematics, Vrije Universiteit Amsterdam}
\email{r.s.hoekzema@vu.nl}
\author[C. Rovi]{Carmen Rovi }
\address{Department of Mathematics \& Statistics, Loyola University Chicago}
\email{crovi@luc.edu}
\author[J. Semikina]{Julia Semikina}
\address{Laboratoire Paul Painlev\'e, Universit\'e de Lille}
\email{iuliia.semikina@univ-lille.fr}

\date{}

\title{A $K$-theory spectrum for cobordism cut and paste groups}

\begin{document}

\begin{abstract} 
Cobordism groups and cut and paste groups of manifolds arise from imposing  two different relations on the monoid of manifolds under disjoint union.  By imposing both relations simultaneously, a cobordism cut and paste group $\SKbar_n$ is defined.  
In this paper, we extend this definition to manifolds with boundary obtaining  the group $\SKbar^{\partial}_n$ and study the relationship of the group to an appropriately defined cobordism group of manifolds with boundary. The main results are the construction of a spectrum that recovers on $\pi_0$ the cobordism cut and paste groups of manifolds with boundary, $\SKbar^{\partial}_n$, and a map of spectra that lifts the canonical quotient map $\SK^{\partial}_n \rightarrow \SKbar^{\partial}_n$.

\end{abstract}

\maketitle

\vspace{-2ex}

\begingroup%
\setlength{\parskip}{\storeparskip}
\tableofcontents
\endgroup%

\setcounter{section}{0}

\section{Introduction}

The classical notion of cut and paste equivalence for closed manifolds was first introduced by Karras, Kreck, Neumann, and Ossa \cite{KKNO}. The cut and paste operation on closed $n$-dimensional manifolds is described by the following procedure. 
Given a manifold, one cuts along a separating codimension 1 submanifold with trivial normal bundle and glues the resulting two pieces back together along an orientation-preserving diffeomorphism to obtain a new manifold.
Manifolds $M$ and $N$ are called $\SK$-equivalent ($\SK$ stands for ``schneiden und kleben'', German for ``cut and paste'') if $N$ can be obtained from $M$ by a sequence of finitely many cut and paste operations. Equivalence classes of $n$-manifolds up to $\SK$-equivalence form the groups $\SK_n$ of cut and paste invariants of manifolds, with the operation given by disjoint union.

In \cite{SKpaper} the notions of $\SK$-equivalence and $\SK$-groups were generalised to the case of manifolds with boundary. 
Allowing boundary made it possible
to formulate a suitable notion of a pointed category with squares $\Mnfldbd_n$, that fits into the framework of $K$-theory with squares due to {\color{blue} \cite{campbell2023algebraic}}
.
Their construction is a modified version of the Thomason's construction of the $K$-theory spectrum for Waldhausen's category, but it allows more flexibility for the category. The distinguished squares for the category $\Mnfldbd_n$ correspond to pushout squares that glue manifolds along a common codimension 0 submanifold, and these can be shown to encode the $\SK^\d$-relations. The $K$-theory spectrum $K^{\square}(\Mnfldbd_n)$ constructed in \cite{SKpaper}, recovers the $\SK^\d_n$ as its zeroth homotopy group.

A different notion of equivalence of manifolds is cobordism. Two $n$-dimensional manifolds $M$ and $N$ are cobordant if their disjoint union forms the boundary of a $(n+1)$-dimensional manifold. Cobordism classes of $n$-manifolds again form a group $\Omega_n$ under disjoint union. The cobordism cut and paste group $\SKbar_n$ is given by quotienting $n$-manifolds by both cobordism and $\SK$-equivalence. For oriented manifolds it was shown in \cite{KKNO} that the only cobordism cut and paste invariant is the signature.

In the current paper we generalise the definition of the $\SKbar$-groups to the case of manifolds with boundary and denote the corresponding group $\SKbar^{\d}_n$. 
In order to define this group we need to quotient by an appropriate notion of cobordism for manifolds with boundary. The usual notion of cobordism for manifolds with boundary allows for cobordisms with corners and free boundary components, but in this setting any two manifolds of the same dimension are cobordant.
To fix this issue we restrict to cobordisms that are trivial on the boundary, meaning that when restricted to the boundary of the in- and outgoing manifolds, the cobordism is diffeomorphic to a cylinder. We show that the groups $\SKbar^{\d}_n$ are related to the classical $\SKbar_n$ groups via an exact sequence analogous to the one found for $\SK^{\d}_n$ in \cite{SKpaper}

   \begin{center}
    $\begin{CD}
     0 @>>> \SKbar_{n} @>\alpha>{[M] \mapsto [M]}> \SKbar^{\partial}_{n} @>\beta>{[N] \mapsto [\partial N]}> C_{n-1} @>>>0
     \end{CD}.$
    \end{center} 
Here $C_{n-1}$ is the group completion of the monoid of orientation preserving diffeomorphism classes of  bounding $(n-1)$-manifolds under disjoint union. 
This is a free abelian group. 

Next we construct a $K$-theory spectrum that recovers the $\SKbar^{\d}_n$ as its zeroth homotopy group.
For this we take inspiration from the $K$-theory with squares construction of {\color{blue} \cite{campbell2023algebraic}} and generalise it in several ways. The input for the $K^{\square}$-construction is a category with squares $\C$, which is a pointed category with two subcategories of morphisms referred to as horizontal (denoted $\hookrightarrow$) and vertical maps (denoted $\rightarrowtail$), and distinguished squares 
$$\begin{tikzcd}[column sep=tiny, row sep=tiny]
    	A\arrow[rr, hook] \arrow[dd, tail] && B\arrow[dd,tail]\\
    	& \square &\\
    	C\arrow[rr, hook] && D
\end{tikzcd}$$
satisfying certain conditions (see \cite{campbell2023algebraic} and \cite{SKpaper}). 

Given a category with squares $\C$, 
 one associates to it a bisimplicial set $N_{\bullet}\C^\bullet$, where the set of $(k,l)$-simplices $N_{k}\C^{(l)}$ is given by the set of all $k \times l$ diagrams 
$$\begin{tikzcd}[column sep=tiny, row sep=tiny]
    	C_{00} \arrow[rr, hook] \arrow[dd, tail] && C_{01} \arrow[rr, hook] \arrow[dd, tail] && \ldots \arrow[rr, hook] && C_{0l} \arrow[dd, tail] \\
    	& \square && \square && \square &\\
    	C_{10} \arrow[rr, hook] \arrow[dd, tail] && C_{11} \arrow[rr, hook] \arrow[dd, tail] && \ldots \arrow[rr, hook] && C_{1l} \arrow[dd, tail]\\
    	& \square && \square && \square &\\
    	\vdots \arrow[dd, tail] && \vdots \arrow[dd, tail] && \ddots && \vdots \arrow[dd, tail]\\
    	& \square && \square && \square &\\
    	C_{k0} \arrow[rr, hook] && C_{k1} \arrow[rr, hook] && \ldots \arrow[rr, hook] && C_{kl},
\end{tikzcd}$$
with face maps given by deleting a corresponding row or column and degeneracy maps given by inserting the copy of the corresponding row or column. The squares $K$-theory space of $\C$ is $$K^\square(\C)\simeq \Omega_O |N_{\bullet}\C^\bullet|,$$ where $\Omega_O$ is the based loop space, based at the distinguished object $O\in N_0\C^{(0)}.$
For our purposes we generalise the construction as follows.
\begin{enumerate}
    \item Instead of having two simplicial directions and distinguished squares we will have three directions($\hookrightarrow$, $\rightarrowtail$, $\rightsquigarrow$) and distinguished cubes as part of the data. It is easy to see that the construction of {\color{blue} \cite{campbell2023algebraic}}
    has a straightforward generalisation for the categories with distinguished $m$-cubes for any $m \geq 2$.
    
    \item The substantial modification is that we do not require morphisms corresponding to different simplicial directions to come from the same category. In our application objects are $n$-dimensional manifolds with boundary, but the three classes of morphisms are of a different nature: two simplicial directions, horizontal ($\hookrightarrow$) and vertical ($\rightarrowtail$), are coming from a particular type of embeddings of manifolds that we refer to as $\SK$-embeddings;  the third (depth) simplicial direction ($\rightsquigarrow$) is given by cobordisms between manifolds that are cylindrical on the boundary. We moreover have $\SK$-embeddings between the cobordisms as part of the data defining a cube.
    
    \item  We take the topology into account  by means of a fourth simplicial direction, in the spirit of \cite{raptis2017parametrized}.
    As a result of this modification, we obtain a topological version of the cut and paste spectrum of manifolds (for the discrete one see \cite{SKpaper}), which also has $\SK_{n}^{\d}$ as its $\pi_0.$ We denote it by $K^{\square}(\mbar).$ 
\end{enumerate}

We construct a certain  quadrisimplicial set $X^*_{\bullet, \bullet, \bullet}$ with $(r, k, l, m)$-simplices given by $k \times l \times m$ diagrams consisting of distinguished cubes fibred over $\Delta^r$ and show that its geometric realisation is an infinite loop space. 
Since the resulting space is connected, its $\pi_0$ does not say anything interesting, and we shift $\pi_0$ by looking at the space of loops instead. We denote the resulting spectrum by $K^\cube(\overline{\text{Mfd}}^\partial_n)$, where the $\cube$ notation reflects the three simplicial directions  given by maps of manifolds that are used in the construction; $\Mnfldbd_n$ refers to the nature of objects; and bar over it is present to emphasize the topology  captured in the fourth simplicial direction. Let us denote by $K_i^\cube(\overline{\text{Mfd}}^\partial_n)$ the $\pi_i$ of the spectrum. 
 \begin{introthm} \label{pi_o thm}
There is an isomorphism $K_0^\cube(\mbar) \cong \SKbar^{\d}_n$.
 \end{introthm}

  From our definition of the quadrisimplicial set $X^*_{\bullet, \bullet, \bullet}$ we automatically deduce a topological analogue of the squares cut and paste spectrum, whose infinite loop space is obtained from a trisimplicial space $X^*_{\bullet, \bullet, 0}$ by taking the loop space of its geometric realisation. We show that the inclusion of the trisimplicial into the quadrisimplicial set 
 provides a categorification of the quotient map $\SK_{n}^\d \rightarrow \SKbar_{n}^\d$.
 
 \begin{introthm} \label{quotient map thm}
   The map of spectra 
   \[
   K^{\square}(\mbar) \to K^{\cube}(\mbar)
   \]
  coming from the inclusion  $X^*_{\bullet, \bullet, 0} \to X^*_{\bullet, \bullet, \bullet}$ induces the canonical quotient map
  \[
  \SK_{n}^\d \rightarrow \SKbar_{n}^\d
  \]
    on the zeroth homotopy groups.
 \end{introthm}

The paper is organised as follows. In Section \ref{SKbarsec} we introduce an appropriate trivial boundary cobordism relation and corresponding cobordism groups $\Omega^\partial_n$ for smooth compact oriented $n$-manifolds with boundary. Then we define $\SKbar^{\partial}$-groups (``cut and paste cobordism groups'') for manifolds with boundary. We relate these new groups to the classical ones for closed manifolds via exact sequences. In Section \ref{Xsec} we construct  a quadrisimplicial set $X^*_{\bullet, \bullet, \bullet}$ in the spirit of the construction of \cite{campbell2023algebraic}, which gives rise to a spectrum $K^{\cube}(\mbar)$. In Section \ref{GammaSec} we verify that  $|X^*_{\bullet, \bullet, \bullet}|$ is indeed an infinite loop space. In Section \ref{pi_0ComputationSec} we prove Theorem \ref{pi_o thm} and finally in Section \ref{induced maps section} we prove  Theorem \ref{quotient map thm}.

\subsection*{Conventions}
All manifolds in this paper will be smooth, oriented and compact. All maps between manifolds/cobordisms will be smooth and orientation preserving.

\subsection*{Acknowledgements}
We would like to thank Jonathan Campbell, Jim Davis, Johannes Ebert, Fabian Hebestreit, Manuel Krannich, Achim Krause, Wolfgang L\"{u}ck, Mona Merling, Thomas Nikolaus,  Oscar Randal-Williams,    George Raptis, Jens Reinhold, Jan Steinebrunner, Wolfgang Steimle, Inna Zakharevich for helpful discussions. We would also like to thank the anonymous referee for very helpful comments
We thank the Mathematics and Statistics Department at Loyola Chicago for hosting us. This research was supported through the program ``Research in Pairs”
by the Mathematisches Forschungsinstitut Oberwolfach in 2020 and 2022, and through the program ``Summer Research in Mathematics'' by the Mathematical Sciences Research Institute. 
The first author was supported by the Emerson Collective and by the Dutch Research Council (NWO) through the grant VI.Veni.212.170.
The second author was  supported by a summer research grant from Loyola University Chicago and by MPIM Bonn.
The third author was funded by the Deutsche Forschungsgemeinschaft (DFG, German Research Foundation) under Germany's Excellence Strategy EXC 2044 –390685587, Mathematics M\"unster: Dynamics–Geometry–Structure.

\section{Cobordism cut and paste groups of manifolds with boundary}\label{SKbarsec}

\subsection{Classical $\SKbar$ groups}
The cut and paste group $\SK_n$ and the cobordism group $\Omega_n$ are both formed by quotienting the monoid of manifolds under disjoint union by the cut and paste relation in the first case, and by the cobordism relation in the second. There are no natural maps from one group to the other in either direction. They can be compared by quotienting the monoid of manifolds by both relations, forming the cobordism cut and paste group $\SKbar_n$. The following pair of short exact sequences from \cite{KKNO} desribe the kernels of the two quotients:
\begin{eqnarray}
F_n \longrightarrow 
&  \Omega_n \longrightarrow & \SKbar_n \label{eqn:twoSES-1}\\
I_n \longrightarrow & \SK_n \longrightarrow & \SKbar_n. \label{eqn:twoSES-2}
\end{eqnarray}

In even dimensions $I_n$ is the subgroup of $\SK_n$ generated by spheres $S^n$. In odd dimensions $I_n$ is zero. The group $F_n \subseteq \Omega_n$ is the subgroup of bordism classes of closed manifolds that fibre over $S^1$. 
It was shown in \cite{KKNO} that the only cut and paste invariants of oriented manifolds are the Euler characteristic $\chi$ and the signature $\sigma$. Noting that $\sigma$ and $\chi$ have the same parity, we have:
\begin{equation*}
  \SK_n \cong  \begin{cases}
    0 & \text{for $n$ odd}\\
    \Z & \text{for $n\equiv 2  \;(\mathrm{mod}\; 4) $ \quad given by $\sfrac{\chi}{2}$}\\
    \Z^2 & \text{for $n\equiv 0  \;(\mathrm{mod}\; 4) $ \quad given by $\left(\frac{\chi - \sigma}{2}, \sigma \right)$.}
  \end{cases}
\end{equation*}

In $\Omega_n$, two manifolds are in the same class if and only if their Stiefel-Whitney and Pontrjagin numbers agree. In particular the signature, which is expressed as a linear combination of Pontrjagin numbers by the Hirzebruch signature theorem, is a cobordism invariant, as well as the Euler characteristic modulo 2 given by the top-dimensional Stiefel-Whitney number. 
In the cobordism cut and paste group only the signature remains
\begin{equation*}
 \hspace{-35pt} \SKbar_n \cong  \begin{cases}
    0 & \text{for $n\not\equiv 0  \;(\mathrm{mod}\; 4) $}\\
    \Z & \text{for $n\equiv 0  \;(\mathrm{mod}\; 4) $ \quad given by $\sigma$.}
  \end{cases}
\end{equation*}

\subsection{Cobordism with trivial boundary} \label{Sec: cob triv boundary}

\begin{definition}
Let $\mM^\d_n$ be the monoid of diffeomorphism classes of smooth compact oriented $n$-dimensional manifolds with boundary under disjoint union. We define the
\textit{trivial boundary cobordism monoid}
$\widetilde{\Omega}^\partial_{n}$ to be $\mM^\d_n$ quotiented by the following relation.
Two classes of manifolds $[M]$ and $[M']$ with boundary are equivalent if there exists a cobordism $W$ with boundary and corners, where the boundary consists of three parts: one that is diffeomorphic to $M$, one that is diffeomorphic to $\bar{M}'$ and one that is diffeomorphic to $\d M \times I \cong  \overbar{\d M'} \times I$. 
Recall that $\bar{M}'$ stands for manifold $M'$ with an opposite orientation.
\end{definition}

\begin{definition}
The \textit{trivial boundary cobordism group}
$\Omega^\partial_n$ is defined to be the group completion of $\widetilde{\Omega}^\partial_{n}$.
\end{definition}

\begin{remk}
The trivial boundary cobordism monoid $\widetilde{\Omega}^\partial_{n}$ is not cancellative.
\end{remk}

 A trivial boundary cobordism corresponds to the trace of a sequence of surgery operations applied to the interior of the ingoing manifold, leaving its boundary unchanged. 
On closed manifolds of dimension $4k$, a surgery operation leaves the signature invariant. The fact that this is also true for a surgery on the interior of a manifold with boundary, follows from applying Wall's non-additivity theorem \cite{wall1969non, kirby2006topology} to the gluing of manifolds with boundary along entire connected components of the boundary.
Wall's non-additivity of the signature features a correction term involving codimension two data that is empty in this situation. The theorem therefore reduces to a version of Novikov additivity \cite[Proposition 7.1]{atiyah1968index} for manifolds with boundary, provided that the inclusion maps of the two pieces take boundary components entirely to the boundary or to the interior (the maps should be a cut and paste embeddings as defined below in Definition \ref{def:SKemb}). Hence the signature for manifolds with boundary is an invariant of trivial boundary cobordism. 


\begin{lemma}\label{Omegapartial}\itshape
 For every $n \geq 1$, $\Omega^\partial_n$ fits in the following exact sequence of abelian groups
\begin{center}
	$\begin{CD}
	0 @>>> F_{n} @> i >> \Omega_{n} @>\alpha>{[M] \mapsto [M]}> \Omega^\partial_{n} @>\beta>{[N] \mapsto [\partial N]}> C_{n-1} @>>>0
	\end{CD}.$
\end{center} 
Here $C_{n-1}$ is the group completion of the monoid of orientation preserving diffeomorphism classes of  bounding $(n-1)$-manifolds under disjoint union, which is a free abelian group, and $F_{n}$ is the subgroup of $\Omega_{n}$ generated by mapping tori. \end{lemma}
\begin{proof}

Two closed manifolds that are cobordant in $\Omega_{n}$ stay
cobordant in $\Omega^\partial_{n}$, hence the map $\alpha$ taking a cobordism class of closed manifolds  to the class containing the same manifiolds in $\Omega^\partial_{n}$ is well-defined. The equivalence relation in $\Omega^\partial_{n}$ preserves the diffeomorphism class of the boundary and therefore the map $\beta$ taking a class of manifolds to the diffeomorphism class of the boundary is a well-defined map.

We show exactness at $\Omega_n$.
A manifold that can be written as a twisted double $M \cup_{\phi} \bar{M}$ is zero in $\Omega^\partial_{n}$ as we have the following relation:
  \[
[M \cup_{\phi} \bar{M}] +   
[M]= [M], \]
as illustrated in Figure \ref{fig:twisted_double} if we set $N$ to equal $M$.
It was shown in \cite{Winkelnkemper} that any twisted double $M \cup_{\phi} \bar{M}$ is canonically cobordant to the mapping torus of $\phi$ 
hence $F_n$ is in the kernel of $\alpha$. 
Consider $[M]$ an element in $\ker \alpha$, where $M$ is a closed manifold.
The signature for manifolds with boundary restricts to the ordinary signature on closed manifolds and therefore $[M] \in \mathrm{Gr}(\mathcal{M}_n)$ has signature zero.
It follows that $[M]$ lies in the image of $F_n$ by \cite[Theorem 1.3]{KKNO}.

Every class in $C_{n-1}$ is the boundary of an $n$-dimensional manifold hence $\widetilde{\beta}:\,\widetilde{\Omega}^\partial_{n} \rightarrow \widetilde{C}_{n-1}$ is surjective, therefore so is the corresponding map on group completions which is $\beta$. 
We show exactness at $\Omega_n^\partial$. It is clear that $\im ~  \alpha \subseteq \ker \beta$. Let $x \in \ker \beta$. We can write $x= [M]-[N],$ where $M, N$ are compact smooth oriented $n$-manifolds with boundary. 
As $\beta(x)= 0$, we have that $\partial M$ and $\partial N$ are diffeomorphic. Choose a diffeomorphism $\phi \colon \partial M \to \partial N.$  Let us view a cylinder on $M \cup _{\phi} \bar{N}$ as a cobordism from $M \cup _{\phi} \bar{N}$ to itself. Consider the outgoing copy of $M \cup _{\phi} \bar{N}$ and choose a one-sided collar in $M$ of the glued boundary $\partial M$. This gives the decomposition of the outgoing boundary of the cylinder as $M \cup (\partial M \times I) \cup _{\phi} \bar{N} $. Finally, this can be viewed as a trivial boundary cobordism from $(M \cup _{\phi} \bar{N}) \bigsqcup N$  to $M$, where the collar is used to produce the boundary piece of cobordism diffeomorphic to $\partial M \times I$. For illustration see again Figure \ref{fig:twisted_double}.
The following relation holds in $\widetilde{\Omega}^\partial_{n}$:
  \[
[M]=[M \cup _{\phi} \bar{N}] + [N]. \]

\begin{figure}[ht]
\hspace{2cm}
\adjustbox{scale=0.45}{

\tikzset{every picture/.style={line width=0.75pt}} 

\begin{tikzpicture}[x=0.75pt,y=0.75pt,yscale=-1,xscale=1]

\draw (386.5,238.96) node  {\includegraphics[width=376.5pt,height=268.43pt]{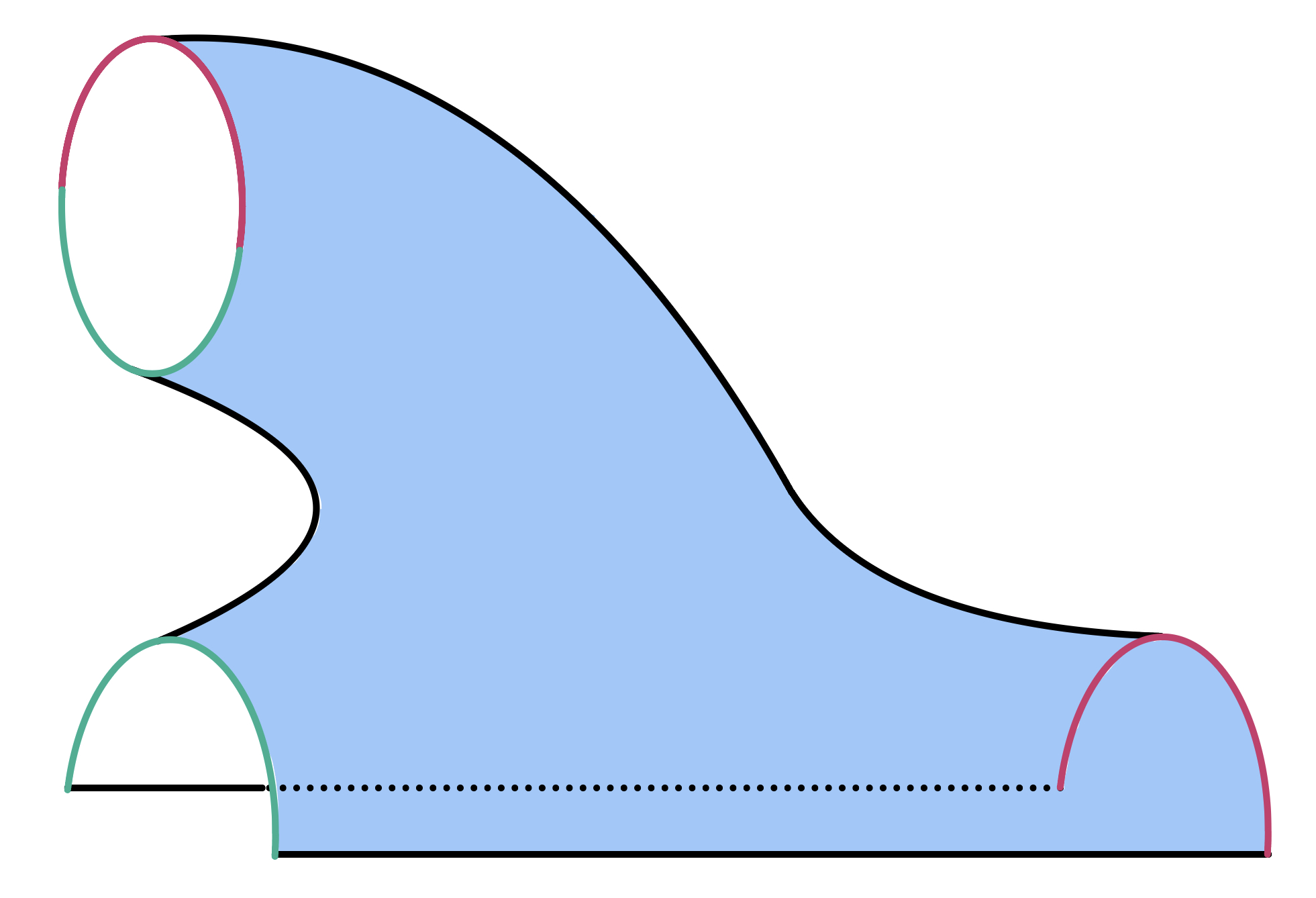}};
\draw    (138,129) -- (137.53,160) ;
\draw [shift={(137.5,162)}, rotate = 270.87] [color={rgb, 255:red, 0; green, 0; blue, 0 }  ][line width=0.75]    (10.93,-3.29) .. controls (6.95,-1.4) and (3.31,-0.3) .. (0,0) .. controls (3.31,0.3) and (6.95,1.4) .. (10.93,3.29)   ;
\draw    (398,418) -- (407.5,418) -- (428.5,418) ;
\draw [shift={(430.5,418)}, rotate = 180] [color={rgb, 255:red, 0; green, 0; blue, 0 }  ][line width=0.75]    (10.93,-3.29) .. controls (6.95,-1.4) and (3.31,-0.3) .. (0,0) .. controls (3.31,0.3) and (6.95,1.4) .. (10.93,3.29)   ;

\draw (124.25,88.99) node [anchor=north west][inner sep=0.75pt]  [font=\LARGE] [align=left] {$\displaystyle M$};
\draw (125.25,171.99) node [anchor=north west][inner sep=0.75pt]  [font=\LARGE] [align=left] {$\displaystyle \overline{N}$};
\draw (115.25,127.99) node [anchor=north west][inner sep=0.75pt]  [font=\Large] [align=left] {$\displaystyle {\textstyle \phi }$};
\draw (403.25,420.99) node [anchor=north west][inner sep=0.75pt]  [font=\Large] [align=left] {$\displaystyle {\textstyle \phi }$};
\draw (627.25,344.99) node [anchor=north west][inner sep=0.75pt]  [font=\LARGE] [align=left] {$\displaystyle M$};
\draw (129.25,334.99) node [anchor=north west][inner sep=0.75pt]  [font=\LARGE] [align=left] {$\displaystyle N$};

\end{tikzpicture}

}
\caption{Relation in $\Omega^{\partial}$}
\label{fig:twisted_double}
\end{figure}
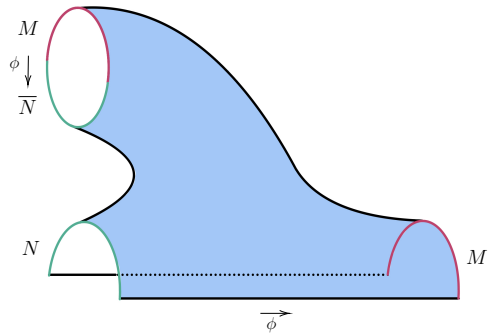

Therefore in $\Omega^\partial_{n}$ we have:
  \[
x=[M]-[N]=[M\cup_{\phi} \bar{N}] \in \im \alpha.\]

\end{proof}

\subsection{$\SKbar$ groups of manifolds with boundary}

We first recall the definition of $\SK^\d$-equivalence from \cite{SKpaper}. A cut and paste operation for manifolds with boundary is  defined as follows: we cut an $n$-dimensional manifold $M$ along a codimension 1 smooth separating submanifold  $\Sigma$ with trivial normal bundle that is disjoint from the boundary of $M$.
We paste back the two pieces together along an orientation preserving diffeomorphism of $\Sigma$. In particular all the cutting boundaries are pasted back together, and the existing boundary of $M$ is left unchanged. We call two manifolds with boundary $\SK^\d$-equivalent if they can be obtained from one another through a finite sequence of cut and paste operations.

In other words, given compact oriented manifolds $M_1, M_2$, closed submanifolds $\Sigma \subseteq \partial M_1$ and $\Sigma' \subseteq \partial M_2$, and orientation preserving diffeomorphisms $\phi, \psi \colon \Sigma \to \Sigma'$, we set
\[ M_1 \cup_{\phi} \bar{M_2} \sim_{\SK^\d} M_1 \cup_{\psi} \bar{M_2}.\]

\begin{definition}
 The \emph{cobordism cut and paste group of manifolds with boundary} $\SKbar^\partial_n$ is defined as the Grothendieck group of diffeomorphism classes of $n$-dimensional manifolds with boundary where addition is given by disjoint union, $Gr(\mathcal{M}^\d_n)$, quotiented by the relation of cobordism with trivial boundary as well as the relation of cutting and pasting of manifolds with boundary. 
\end{definition}

\begin{lemma}\label{SKbardelta}\itshape
The cobordism cut and paste group of manifolds with boundary $\SKbar^\partial_n$ is isomorphic to the trivial boundary cobordism group $\Omega^\partial_n$.
\end{lemma}

\begin{proof}

Consider the following diagram
	\[ \begin{tikzcd}
	\Omega_n/F_n \arrow[r, hook,"\bar{\alpha}"]
	\arrow[d,"\gamma", "\cong"'] & \Omega^\partial_n \arrow[r, two heads, "\beta"]\arrow[d, two heads] & C_{n-1} \arrow{d}{\id} \\
	\SKbar_n \arrow{r}{\alpha'} & \SKbar^\partial_n \arrow{r}{\beta'} & C_{n-1},
	\end{tikzcd}
	\]
where $\alpha'$ sends a cobordism cut and paste class of closed manifolds to their class as manifolds with boundary, and $\beta'$ takes a class of manifolds to their boundary, which is preserved up to diffeomorphism by both cut and paste and cobordism relations. 
The map $\gamma$ is the canonical isomorphism coming from Equation (\ref{eqn:twoSES-1}) and the map from $\Omega^\partial_n$ to $\SKbar^\partial_n$ is the canonical surjection. Both squares clearly commute.

The top line is a short exact sequence by Lemma \ref{Omegapartial}. 
As the cut and paste relation for manifolds with boundary preserves signature of manifolds with boundary, the signature map is a splitting for $\alpha'$, as it is for $\bar{\alpha}$. Hence $\alpha'$ is injective. Both the cut and paste and trivial boundary cobordism relations preserve the boundary up to diffeomorphism hence $\beta'$ is surjective. 
Let $x$ be an element in the kernel of $\beta'$, and $\bar{x}$ a preimage of $x$ in $\Omega_n^\partial$. Then $\bar{x}$ is in the kernel of $\beta$ hence we can find an element $y$ in $\Omega_n/F_n$ that maps to it. Hence $x=\alpha'(\gamma(y))$, so the lower sequence is exact. 
By the five lemma, the middle groups are isomorphic.
\end{proof}

\subsection{Fibrations of cobordism categories} \label{section:fibrations}

We can also consider trivial boundary cobordism from the point of view of
a fibre sequence of cobordism categories that includes into Genauer's fibre sequence for the cobordism category of manifolds with boundary \cite{GMTW,genauer2012cobordism}. We will make this precise in the Proposition below. Let $\Cob_{n+1}^{\partial_\text{all}}$ be the weakly unital topological cobordism category with boundary, with objects $n$-dimensional manifolds with boundary and morphisms given by $(n+1)$-dimensional cobordisms with boundary, as defined in \cite{steimle2021additivity}.
As a subcategory of $\Cob_{n+1}^{\partial_\text{all}}$, we consider the weakly unital category $\Cob_{n+1}^{\partial_\text{triv}}$ of $n$-dimensional manifolds with boundary, with morphisms given by $(n+1)$-dimensional cobordisms that are trivial on the boundary of the objects, meaning that in the cobordism direction, their boundary is diffeomorphic to $\partial M \times [0,1]$ if $M$ is the source (or equivalently target) object.
The boundary map from $\Cob_{n+1}^{\partial_\text{all}}$ to $\Cob_{n}$ in \cite{steimle2021additivity} restricted to $\Cob_{n+1}^{\partial_\text{triv}}$ then has image given by the subcategory $\Map^\varnothing_{n}$ defined below.

\begin{definition}
Let $\partial M_0$ and $\partial M_1$ be nullbordant $(n-1)$-dimensional manifolds. We call a trivial cobordism (a cobordism that is diffeomorphic to a cylinder) from $\partial M_0$ to $\partial M_1$ \emph{extendable} if there are nullbordisms $M_0, M_1$ of $\partial M_0$ and $\partial M_1$ respectively, and an $(n+1)$-dimensional trivial boundary cobordism $W$ with $M_0 = \partial W_{in}, M_1 = \partial W_{out}$, where $ \partial W_{in}$ is the incoming boundary and $\partial W_{out}$ is the outgoing boundary of $W$ (see Figure \ref{Fig:trivial cob}).
\end{definition}
The category $\Map^\varnothing_{n}$ has objects $(n-1)$-dimensional nullbordant manifolds and morphisms $n$-dimensional extendable trivial cobordisms.
This is a subcategory of $\Cob^\varnothing_{n}$, the full subcategory on bounding manifolds inside $\Cob_{n}$.

\begin{prop}\label{cobcats}\itshape
The following diagram is a pullback diagram of weakly unital, locally fibrant topological categories and weakly unital continuous functors
\[ \begin{tikzcd}
\Cob_{n+1}^{\partial_\text{triv}}  \arrow[r,  "\partial"]\arrow[d] & \Map^\varnothing_{n}
\arrow{d}{} \\
\Cob_{n+1}^{\partial_\text{all}} \arrow{r}{\partial}  & \Cob^\varnothing_{n}.
	\end{tikzcd}\]
\end{prop}
\begin{proof}
First note that $\Map^\varnothing_{n}$ is weakly unital in the sense of \cite{steimle2021additivity} section 2, and the inclusion into $\Cob^\varnothing_{n}$ is weakly unital continuous.
The morphism space of the category $\Map^\varnothing_{n}$ consists of a union of connected components of the morphism space of $\Cob^\varnothing_{n}$, hence it is locally fibrant in the sense of \cite{steimle2021additivity} as $\Cob_{n}$ is.

On the level of objects, we have the following diagram which is clearly a pullback:
\begin{center}
    	\begin{tikzcd}
    		\{ \text{$n$-dim mfds with boundary}\} \arrow[rr, "\d"] \arrow[dd, "\id"] &&  	\{ \text{$(n-1)$-dim closed nullbordant mfds}\}\arrow[dd, "\id"]\\
    		&&\\
    		\{ \text{$n$-dim mfds with boundary}\} \arrow[rr, "\d"] && \{ \text{$(n-1)$-dim closed nullbordant mfds}\}.
    	\end{tikzcd}
\end{center}
    	The morphism spaces of the above diagram of categories also form a pullback diagram. It is enough to check the pullback property for the morphisms spaces between any two ($n+1$)-dimensional manifolds $M$  and $M'$. In this case, the diagram is:
 \begin{center}
\begin{tikzcd}
    	\parbox{6cm}{Space of cobordisms between $M$ and $M'$ that have boundary diffeomorphic to  $\d M \times [0,1]$}
    		 \arrow[rr, "\d"] \arrow[dd, "incl"] &&
    		 \parbox{5cm}{Space of extendable trivial cobordisms from $\d M$ to $\d M'$}	
    		 \arrow[dd, "incl"]\\
    		& &\\
    		\parbox{6cm}{Space of cobordisms with boundary (in Genauer's sense) between $M$ and $M'$ } \arrow[rr, "\d"] && \parbox{5cm}{Space of cobordisms from $\d M$ to $\d M'$}
    	\end{tikzcd}.
\end{center}
This is indeed a pullback diagram in the category of topological spaces.
\end{proof}

We apply Theorem 2.3 in \cite{steimle2021additivity} to the pullback square of categories in Prop. \ref{cobcats}. By Lemma 4.4 in \cite{steimle2021additivity}, the boundary map from $\Cob_{n+1}^{\partial_\text{all}}$ to $\Cob_{n}$ is a level cartesian and cocartesian fibration. 
We hence obtain a homotopy pullback square of classifying spaces
	\[ \begin{tikzcd}
B\Cob_{n+1}^{\partial_\text{triv}}  \arrow[r,  "B\partial"]\arrow[d] & B\Map^\varnothing_{n}
\arrow{d}{} \\
B\Cob_{n+1}^{\partial_\text{all}} \arrow{r}{B\partial}  & B\Cob^\varnothing_{n}.
	\end{tikzcd}\]
The lower line of this diagram is the well-known Genauer fibration with fibre $B\Cob_{n+1}$, therefore the upper line is also a fibration with homotopy equivalent fibre: 
\begin{center}
	$\begin{CD}
B\Cob_{n+1} @>>> B\Cob_{n+1}^{\partial_\text{triv}} @>>> B\Map^\varnothing_{n}
	\end{CD}.$
\end{center} 
	By construction this fibration is surjective.

The fibration above gives rise to the following long exact sequence in homotopy: 
\[
\pi_1 B\Cob_{n+1} \longrightarrow \pi_1 B\Cob_{n+1}^{\partial_\text{triv}} \longrightarrow
\pi_1 B\Map^\varnothing_{n} \longrightarrow
\Omega_{n}\longrightarrow \widetilde{\Omega}^\partial_{n} \longrightarrow 
\widetilde{D}_{n-1}
\longrightarrow 0,
\] 
where $\widetilde{D}_{n-1}$ is the monoid of nullbordant oriented $(n-1)$-dimensional manifolds up to extendable trivial cobordism and
$\widetilde{\Omega}^\partial_{n}$ is the trivial boundary cobordism monoid.
On $\pi_0$ we have based the sets at $\varnothing$ to define the exact sequence, and note that the maps are also monoid maps. 
The group $\pi_1 B\Cob_{n+1}$ is also known as the controlled cut and paste group $\SKK_{n+1}$ (see for example \cite{KKNO}). 
It is clear that the map of monoids $\Omega_{n}\longrightarrow \widetilde{\Omega}^\partial_{n}$ is injective hence we have the last three terms as a short exact sequence of monoids.

\section{ Cubes of manifolds with boundary}
 \label{Xsec}

 In what follows, we will define a trisimplicial set $X^0_{\bullet, \bullet,\bullet}$ of cubes of manifolds. Two of three simplicial directions are defined the same way as in the bisimplicial construction used in the K-theory of manifolds from \cite{SKpaper}. These two directions encode the cut and paste data. 
The third direction corresponds to cobordisms that fit into our description of trivial boundary cobordism (See Section \ref{Sec: cob triv boundary} and Figure \ref{Fig:trivial cob}).

On the faces in the cobordism direction we have a cut and paste square of the cobordisms themselves that extends the cut and paste squares of manifolds on the front and the back of the cube.
The cobordisms with boundary will have collars both in the composition direction (of width $\varepsilon_c\rightarrow 0$) as well as in the trivial cobordism boundary direction (of width $\varepsilon_\partial\rightarrow 0$). 

 Instead of directly giving this trisimplicial set a topology, we add a fourth simplicial direction encoding topology and view the initial trisimplicial set as a value at $[0]$ of a simplicial object in the category of trisimplicial sets. 
This approach was introduced in \cite{raptis2017parametrized} to define parametrized cobordism categories and is also applied in \cite[Section 4]{steimle2021additivity}.
In the new simplicial direction in level $r$, we will have (0,0,0)-simplices given by fiber bundles $E \to \Delta^r$ with fibers being manifolds with boundary and generic trisimplicial datum described further.

\begin{figure}[htbp]

\hspace{1.5cm}
\adjustbox{scale=0.55}{

\tikzset{every picture/.style={line width=0.75pt}} 

\begin{tikzpicture}[x=0.75pt,y=0.75pt,yscale=-1,xscale=1]

\draw (335.75,418.5) node  {\includegraphics[width=421.13pt,height=226.5pt]{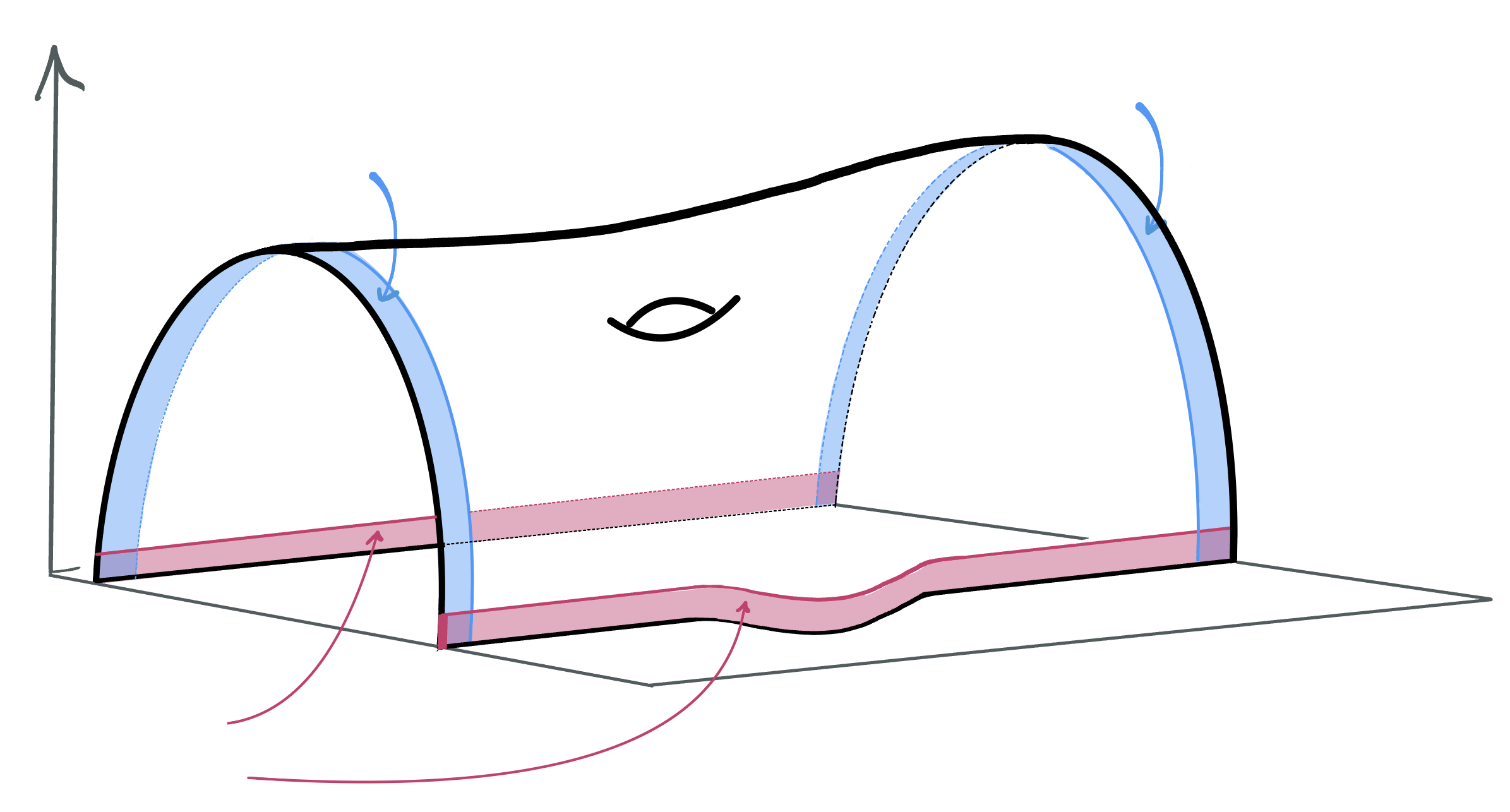}};

\draw (276.25,277.99) node [anchor=north west][inner sep=0.75pt]  [font=\LARGE] [align=left] {$\displaystyle W$};
\draw (93.25,276.99) node [anchor=north west][inner sep=0.75pt]  [font=\large] [align=left] {$\displaystyle \mathbb{R}^{+}$};
\draw (566.25,455.99) node [anchor=north west][inner sep=0.75pt]  [font=\large] [align=left] {$\displaystyle \mathbb{R}^{\infty }$};
\draw (510.25,476.99) node [anchor=north west][inner sep=0.75pt]  [font=\normalsize] [align=left] {$\displaystyle l$};
\draw (214.25,510.99) node [anchor=north west][inner sep=0.75pt]  [font=\normalsize] [align=left] {$\displaystyle 0$};
\draw (151.25,302.99) node [anchor=north west][inner sep=0.75pt]  [font=\large] [align=left] {$\displaystyle \partial W_{i}{}_{n} \times \epsilon _{c}$};
\draw (440.25,275.99) node [anchor=north west][inner sep=0.75pt]  [font=\large] [align=left] {$\displaystyle \partial W_{o}{}_{u}{}_{t} \times \epsilon _{c}$};
\draw (40.25,535.99) node [anchor=north west][inner sep=0.75pt]  [font=\large] [align=left] {$\displaystyle \partial W_{b}{}_{d}{}_{y} \times \epsilon _{\partial }$};

\end{tikzpicture}}
\caption{The $\epsilon_c$ and $\epsilon_{\partial}$ collars on a cobordism with trivial boundary.}
\label{Fig:trivial cob}
\end{figure}
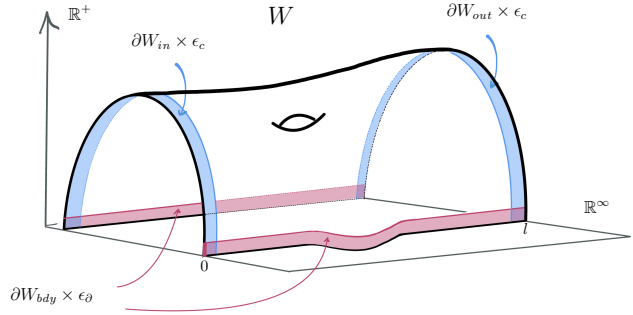

\subsection{The  simplicial set $X^*_{0,0,0}$} of manifolds with boundary

\begin{definition}\label{def:neat}
     Let $M$ be a smooth oriented submanifold of $\R_+ \times \R^\infty$, with $\partial M = M\cap (\{0\} \times \R^\infty)$. 
    Then we say that $M$ is \emph{neat} (in the sense of \cite{genauer2012cobordism,steimle2021additivity}) if it is cylindrical near the boundary, i.e. there exists an $\varepsilon_\partial >0$ such that \[
M\cap ([0,\varepsilon_\partial)\times \R^\infty) = [0,\varepsilon_\partial)\times \partial M.   
\]
We say an \emph{embedding} $e$ of $\varepsilon_\d$-neat manifold $M$ into $\R_+ \times \R^\infty$ is \textit{$\varepsilon_\d$-neat} if the following diagram commutes
\[ \begin{tikzcd}\,
[0,\varepsilon_\d)  \times \d M \arrow[r, hook, "incl"]\arrow[d, " \id \times  e|_{\d M}"] & M
\arrow[d, "e"] \\ \,
[0, \varepsilon_\d) \times \R^\infty  \arrow[r, hook] &  \R^\infty \times \R_+.
\end{tikzcd}\]

\end{definition}

We now define the simplicial set $X^*_{0,0,0}$ to be, on level $*=r$ the set of fibre bundles $E \mapsto \Delta^r$ with fibres being smooth manifolds with boundary, which is fiberwise $\varepsilon_\d$-neatly embedded into $\Delta^r \times \R_+\times\R^\infty$ for some $\varepsilon_\d>0$.

\subsection{ Simplicial sets $X^*_{1,0,0}$ and $X^*_{0,1,0}$} of 
SK-embeddings

\begin{definition}\label{def:SKemb}
Let $M$ and $N$ be neat $n$-dimensional manifolds. We define a \emph{cut and paste embedding} or \emph{$\SK$-embedding} from $M$ to $N$ to be a smooth orientation-preserving embedding of manifolds (not necessarily neat) that satisfies the following conditions. Each boundary component of $M$ maps either entirely to $\d N$, or entirely to the interior of $N$ as a submanifold with trivial normal bundle.
\end{definition}


 For $E, E'\in X^r_{0,0,0}$ manifold bundles over $\Delta^r$, let $\CEmb(E,E')$ denote the set of bundle maps from $E$ to $E'$ which are fiberwise SK-embeddings, which we refer to as \emph{bundle $\SK$-embeddings}.
We define the sets $X^r_{1,0,0}$ and $X^r_{0,1,0}$ to be all bundle $\SK$-embeddings between the corresponding source and target, i.e.,
\[X^r_{1,0,0} = X^r_{0,1,0} := 
\bigsqcup_{X_{0,0,0}^r \times X_{0,0,0}^r} \CEmb(E,E').\]

\subsection{The  simplicial set $X^{*}_{0,0,1}$} of cobordisms with trivial boundary
In the third simplicial direction, (0,0,1)-simplices are given by trivial boundary cobordisms. Given two  $\varepsilon_\d$-neat manifolds $M, M' \in X_{0,0,0}$, a  $(0,0,1)$-simplex between them (i.e., a simplex with a source $M$ and target $M'$) is given by a neat cobordism defined as follows. For $l>0$ and $\varepsilon_c, \varepsilon_\d >0$, an \textit{$(\varepsilon_c, \varepsilon_\d)$-neat} cobordism $(l,W)$ from $M$ to $M'$  is an $(n+1)$-dimensional submanifold (with corners) of $[0,l] \times \R_+\times\R^\infty$, with boundary consisting of three pieces $\d W_{in}=M$, $\d W_{out}=M'$, and $\d W_{bdy}$, that are the intersections of $W$ with $\{0\}\times \R_+\times\R^\infty$, $\{l\}\times \R_+\times\R^\infty$, and $[0,l]\times \{0\}\times\R^\infty$ respectively, adhering to the following conditions. 
We have that $W$ is $\varepsilon_c$-cylindrical ($c$ for composition) near the in- and out-components of the boundary, i.e.,
\[
W\cap ([0,\varepsilon_c)\times \R_+\times \R^\infty) =
[0,\varepsilon_c) \times M.   
\]
\[
W\cap ((l-\varepsilon_c, l]\times \R_+\times \R^\infty) =
(l-\varepsilon_c, l] \times M'.   
\]
Moreover $W$ is $\varepsilon_\d$-cylindrical near the $\d W_{bdy}$ component
\[
\tau_{(1 \, 2)} (W\cap ([0, l]\times [0, \varepsilon_\d]\times \R^\infty)) = [0, \varepsilon_\d] \times \d W_{bdy},\]
where $\tau_{(1 \, 2)}$ swaps the first two coordinates.
Lastly, the cobordism is required to be diffeomorphic to a cylinder on $\d W_{bdy}$, i.e.,
$$\partial W_{bdy} = W|_{[0,l]\times\{0\}\times\R^\infty}  \cong [0,l] \times \partial M.$$

We now define the set $\overline{X}^*_{0,0,1}$ to be, on level $*=r$, a map $l:\Delta^r \to  (0, \infty)$ and a fiber bundle over $\Delta^r$ with fibers compact, smooth $n+1$ trivial boundary cobordisms fiberwise $(\varepsilon_c, \varepsilon_\d)$-neatly embedded into $\Delta^r \times [0,l] \times \R_+\times\R^\infty$ for some $\varepsilon_c>0$, $\varepsilon_\d>0$.



 We then define
$$X^r_{0,0,1} = X^r_{0,0,0} \; \;\sqcup \;\;\overline{X}^*_{0,0,1}. $$
The first term can be thought of as cobordisms of length zero, which are the image of the degeneracy map from $X_{0,0,0}$ to $X_{0,0,1}$. The second summand encodes cobordism bundles of non-zero length.


\subsection{The  simplicial set $X^*_{1,1,0}$} of cut and paste squares \label{X_110}

This set is given by the set of cut and paste squares that are distinguished squares of the category with squares $\Mnfldbd_n$ from \cite{SKpaper}. Recall that those squares are given by the following commutative diagrams
\[
       \begin{tikzcd}
    	M_0	\arrow[r, hook]\arrow[d, tail] & M_1 \arrow[d, tail]\\
    	M_2 \arrow[r, hook] & M_1 \cup_{M_0} M_2,
    	\end{tikzcd} 
\]
where maps are $\SK$-embeddings and the diagram is a pushout square in the category of spaces, i.e. the pushout $ M_1 \cup_{M_0} M_2$ is a smooth manifold.

 On the simplicial level $*=r$ the set $X^r_{1,1,0}$ is given by the set of squares of manifold bundles over $\Delta^r$, i.e.  the set of diagrams of total spaces of bundles
 $$
\begin{tikzcd}
E_0	\arrow[r, hook] \arrow[d, tail] & E_1 \arrow[d, tail]\\
E_2 \arrow[r, hook] & E_3,
\end{tikzcd} 
$$
 which over every point of $\Delta^r$ restrict to the distinguished squares of manifolds as above.


\subsection{ Simplicial sets $X^*_{1,0,1}$ and $X^*_{0,1,1}$} of SK-embeddings of cobordisms

Analogously to how we defined SK-embeddings of manifolds with boundary, we now define SK-embeddings of neat cobordisms. 
\begin{definition}
 Let $V$ and $W$ be neat $(n+1)$-cobordisms in $[0,l] \times \R_+\times\R^\infty$. An SK-embedding $F$ from $V$ to $W$ is a smooth orientation preserving embedding such that the following holds. 
 \begin{enumerate}
     \item $F$ sends $\d V_{in}, \d V_{out}$ to $\d W_{in}, \d W_{out}$ correspondingly; the restriction maps $F_{| \d V_{in}}$ are $F_{| \d V_{out}}$ are SK-embeddings of neat manifolds with boundary.
     
     \item $F$ preserves the collars at in- and outgoing boundary components, i.e., $F$ is cylindrical near $\d V_{in}, \d V_{out}$.
     
     \item $F$ sends every component of $\d V_{bdy}$ either entirely to $\d W_{bdy}$ or entirely (apart from the corner points) to the interior of $W$ with the image being a submanifold of $W$ with trivial normal bundle.
 \end{enumerate}

Figure \ref{cob-inclusions} shows an example of an SK-embedding of cobordisms. 

\begin{figure}[htbp]

\hspace{2cm}
\adjustbox{scale=0.75}{

\tikzset{every picture/.style={line width=0.75pt}} 

\begin{tikzpicture}[x=0.75pt,y=0.75pt,yscale=-1,xscale=1]

\draw (336.5,407.99) node  {\includegraphics[width=466.5pt,height=143.23pt]{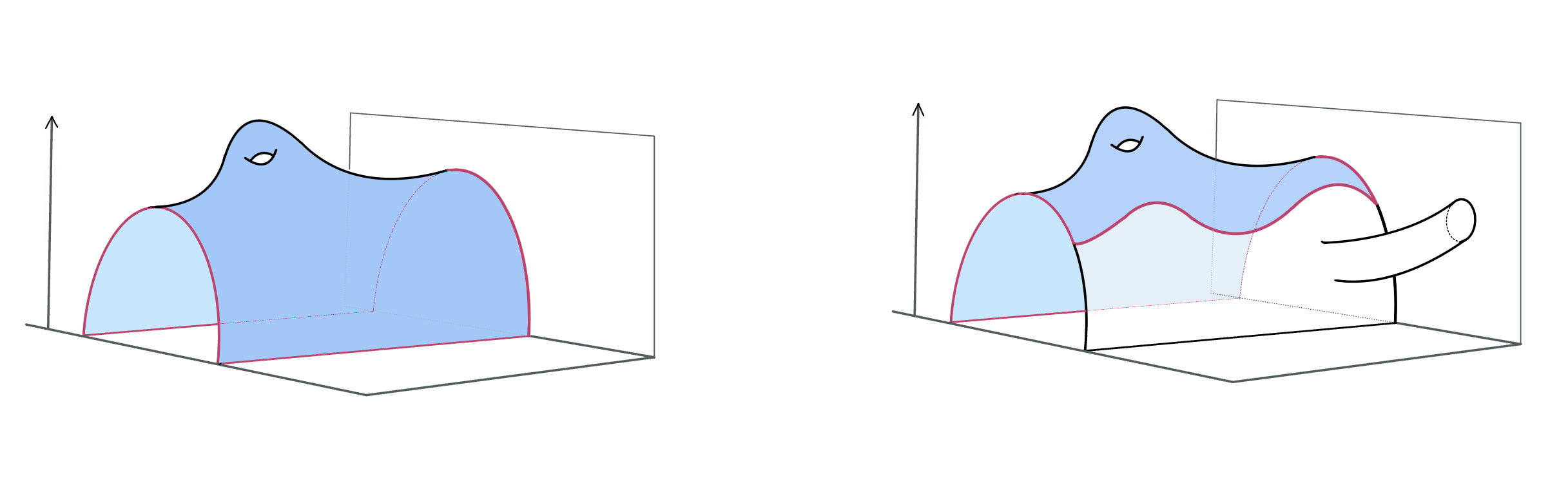}};
\draw    (312.5,416.5) -- (347.5,416.5) ;
\draw [shift={(349.5,416.5)}, rotate = 180] [color={rgb, 255:red, 0; green, 0; blue, 0 }  ][line width=0.75]    (10.93,-3.29) .. controls (6.95,-1.4) and (3.31,-0.3) .. (0,0) .. controls (3.31,0.3) and (6.95,1.4) .. (10.93,3.29)   ;

\draw (49.25,347.99) node [anchor=north west][inner sep=0.75pt]  [font=\small] [align=left] {$\displaystyle \mathbb{R}^{+}$};
\draw (459.25,326.99) node [anchor=north west][inner sep=0.75pt]  [font=\normalsize] [align=left] {$\displaystyle W$};
\draw (118.25,331.99) node [anchor=north west][inner sep=0.75pt]  [font=\normalsize] [align=left] {$\displaystyle V$};
\draw (395.25,347.99) node [anchor=north west][inner sep=0.75pt]  [font=\small] [align=left] {$\displaystyle \mathbb{R}^{+}$};
\draw (109.25,457.99) node [anchor=north west][inner sep=0.75pt]  [font=\footnotesize] [align=left] {$\displaystyle 0$};
\draw (450.25,451.99) node [anchor=north west][inner sep=0.75pt]  [font=\footnotesize] [align=left] {$\displaystyle 0$};
\draw (577.25,440.99) node [anchor=north west][inner sep=0.75pt]  [font=\footnotesize] [align=left] {$\displaystyle l$};
\draw (230.25,447.99) node [anchor=north west][inner sep=0.75pt]  [font=\footnotesize] [align=left] {$\displaystyle l$};
\draw (322.25,392.99) node [anchor=north west][inner sep=0.75pt]  [font=\normalsize] [align=left] {$\displaystyle F$};

\end{tikzpicture}

}
\caption{SK embedding of cobordisms}
\label{cob-inclusions}
\end{figure}
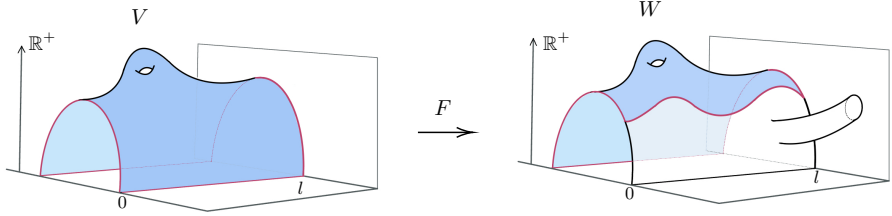

\end{definition}

 For $D,D' \in X^r_{0,0,1}$ trivial cobordism bundles over $\Delta^r$, let $\CEmb(D,D')$ denote the set of bundle maps from $D$ to $D'$ which are fiberwise SK-embeddings of trivial boundary cobordisms, preserving the length $l$ fiberwise.
We define the sets $X^r_{1,0,1}$ and $X^r_{0,1,1}$ to be all bundle $\SK$-embeddings between the corresponding source and target trivial cobordism bundles, i.e.,
\begin{eqnarray}
X^r_{0,1,1} = X^r_{1,0,1} &:= &
\bigsqcup_{E,E' \in X^r_{0,0,0}}
\CEmb(E,E')\nonumber \\
&& \sqcup \;\; \bigsqcup_{D,D' \in X_{0,0,1}^r} \CEmb(D,D').
\nonumber
\end{eqnarray}

\subsection{The  simplicial set $X^*_{1,1,1}$} of cubes of manifolds

We define a cut and paste square of trivial boundary cobordism as 
a quadruple $(F_1, F_2, F_3, F_4)$, where $F_i$ are SK-embeddings of cobordisms of the same length that fit into a commutative square, which is a pushout square on the level of topological spaces
$$\begin{tikzcd}[row sep=scriptsize,column sep=scriptsize]
W_0 \arrow[rr, Rightarrow, "F_1"] \arrow[dd, Rightarrow, "F_2"] && W_1
\arrow[dd, Rightarrow, "F_4"] \\
&&&&\\
W_2 \arrow[rr, Rightarrow, "F_3"] &&  W_1 \cup_{W_0} W_2.
\end{tikzcd}$$
 On the simplicial level $*=r$ the set $X^r_{1,1,1}$ is then given by the set of squares of cobordism bundles over $\Delta^r$, i.e. set of the diagrams 
$$
\begin{tikzcd}
D_0	\arrow[r, hook] \arrow[d, tail] & D_1 \arrow[d, tail]\\
D_2 \arrow[r, hook] & D_3,
\end{tikzcd} 
$$
 such that over each point of $\Delta^r$ this restricts to a cut and paste square of trivial boundary cobordisms as above. Fiberwise we have a cube of manifolds, trivial boundary cobordisms, and $\SK$-embeddings of cobordisms extending the $\SK$-embeddings on manifolds at the in- and outgoing boundary as illustrated below.

\begin{adjustbox}{scale=0.7, center}

\tikzset{every picture/.style={line width=0.75pt}} 

\begin{tikzpicture}[x=0.75pt,y=0.75pt,yscale=-1,xscale=1]

\draw    (328.28,72.27) -- (390.77,72) -- (404.48,72.03) -- (426.26,72.06) -- (463.99,72.24)(328.3,75.27) -- (390.77,75) -- (404.47,75.03) -- (426.25,75.06) -- (463.98,75.24) ;
\draw [shift={(471.98,73.77)}, rotate = 180.26] [color={rgb, 255:red, 0; green, 0; blue, 0 }  ][line width=0.75]    (12.02,-3.62) .. controls (7.65,-1.54) and (3.64,-0.33) .. (0,0) .. controls (3.64,0.33) and (7.65,1.54) .. (12.02,3.62)   ;
\draw    (291.24,128.92) -- (291.22,217.91)(288.24,128.92) -- (288.22,217.91) ;
\draw [shift={(289.71,225.91)}, rotate = 270.02] [color={rgb, 255:red, 0; green, 0; blue, 0 }  ][line width=0.75]    (10.93,-3.29) .. controls (6.95,-1.4) and (3.31,-0.3) .. (0,0) .. controls (3.31,0.3) and (6.95,1.4) .. (10.93,3.29)   ;
\draw    (499.54,97.49) -- (501.33,216.27)(496.54,97.54) -- (498.33,216.32) ;
\draw [shift={(499.95,224.3)}, rotate = 269.14] [color={rgb, 255:red, 0; green, 0; blue, 0 }  ][line width=0.75]    (10.93,-3.29) .. controls (6.95,-1.4) and (3.31,-0.3) .. (0,0) .. controls (3.31,0.3) and (6.95,1.4) .. (10.93,3.29)   ;
\draw    (380,200) -- (413.5,200) -- (450,200) ;
\draw [shift={(380,200)}, rotate = 0] [color={rgb, 255:red, 0; green, 0; blue, 0 }  ][line width=0.75]      (0,-7.83) .. controls (-2.16,-7.83) and (-3.91,-6.07) .. (-3.91,-3.91) .. controls (-3.91,-1.75) and (-2.16,0) .. (0,0) ;
\draw    (466,200) -- (490,200) ;
\draw    (509,200) -- (538,200) ;
\draw [shift={(540,200)}, rotate = 180] [color={rgb, 255:red, 0; green, 0; blue, 0 }  ][line width=0.75]    (10.93,-3.29) .. controls (6.95,-1.4) and (3.31,-0.3) .. (0,0) .. controls (3.31,0.3) and (6.95,1.4) .. (10.93,3.29)   ;
\draw    (462,238.5) -- (474,238.5)(462,241.5) -- (474,241.5) ;
\draw [shift={(482,240)}, rotate = 180] [color={rgb, 255:red, 0; green, 0; blue, 0 }  ][line width=0.75]    (10.93,-3.29) .. controls (6.95,-1.4) and (3.31,-0.3) .. (0,0) .. controls (3.31,0.3) and (6.95,1.4) .. (10.93,3.29)   ;
\draw    (323,238.5) -- (452,238.5)(323,241.5) -- (452,241.5) ;
\draw    (360,127) -- (360,188) ;
\draw [shift={(360,190)}, rotate = 270] [color={rgb, 255:red, 0; green, 0; blue, 0 }  ][line width=0.75]    (10.93,-3.29) .. controls (6.95,-1.4) and (3.31,-0.3) .. (0,0) .. controls (3.31,0.3) and (6.95,1.4) .. (10.93,3.29)   ;
\draw    (360,82) -- (360,114) ;
\draw    (360,64) -- (360,68) ;
\draw [shift={(360,64)}, rotate = 270] [color={rgb, 255:red, 0; green, 0; blue, 0 }  ][line width=0.75]    (10.93,-3.29) .. controls (6.95,-1.4) and (3.31,-0.3) .. (0,0) .. controls (3.31,0.3) and (6.95,1.4) .. (10.93,3.29)   ;
\draw    (290.5,100) -- (290.5,116)(287.5,100) -- (287.5,116) ;

\draw (238,121) node    {$A$};
\draw (360,39.5) node    {$A'$};
\draw (239.87,300.86) node    {$C$};
\draw (556,39.5) node    {$B'$};
\draw (457.5,120.5) node    {$B$};
\draw (360.5,199.5) node    {$C'$};
\draw (555.5,199.5) node    {$D'$};
\draw (458.27,299.86) node    {$D$};
\draw (266.82,62.34) node [anchor=north west][inner sep=0.75pt]    {$W_{0}$};
\draw (480.2,64.25) node [anchor=north west][inner sep=0.75pt]    {$W_{1}$};
\draw (264.82,236.35) node [anchor=north west][inner sep=0.75pt]    {$W_{2}$};
\draw (521.39,237.64) node [anchor=north west][inner sep=0.75pt]    {$W_{1} \ \cup _{W_{0}} \ W_{2}$};
\draw (398.85,51.82) node [anchor=north west][inner sep=0.75pt]    {$F_{1}$};
\draw (267.2,152.94) node [anchor=north west][inner sep=0.75pt]    {$F_{2}$};
\draw (510,142.4) node [anchor=north west][inner sep=0.75pt]    {$F_{4}$};
\draw (398.42,216.77) node [anchor=north west][inner sep=0.75pt]    {$F_{3}$};
\draw    (380,39.5) -- (543,39.5) ;
\draw [shift={(545,39.5)}, rotate = 180] [color={rgb, 255:red, 0; green, 0; blue, 0 }  ][line width=0.75]    (7.65,-2.3) .. controls (4.86,-0.97) and (2.31,-0.21) .. (0,0) .. controls (2.31,0.21) and (4.86,0.98) .. (7.65,2.3)   ;
\draw [shift={(380,39.5)}, rotate = 0] [color={rgb, 255:red, 0; green, 0; blue, 0 }  ][line width=0.75]      (0,-7.83) .. controls (-2.16,-7.83) and (-3.91,-6.07) .. (-3.91,-3.91) .. controls (-3.91,-1.75) and (-2.16,0) .. (0,0) ;
\draw    (251.5,111.98) .. controls (251.96,109.67) and (253.35,108.74) .. (255.66,109.2) .. controls (257.97,109.67) and (259.36,108.74) .. (259.82,106.43) .. controls (260.27,104.12) and (261.66,103.19) .. (263.97,103.65) .. controls (266.28,104.11) and (267.67,103.18) .. (268.13,100.87) .. controls (268.59,98.56) and (269.98,97.63) .. (272.29,98.09) .. controls (274.6,98.56) and (275.99,97.63) .. (276.45,95.32) .. controls (276.9,93.01) and (278.29,92.08) .. (280.6,92.54) .. controls (282.91,93) and (284.3,92.07) .. (284.76,89.76) .. controls (285.22,87.45) and (286.61,86.52) .. (288.92,86.98) .. controls (291.23,87.45) and (292.62,86.52) .. (293.08,84.21) .. controls (293.53,81.9) and (294.92,80.97) .. (297.23,81.43) .. controls (299.54,81.89) and (300.93,80.96) .. (301.39,78.65) .. controls (301.85,76.34) and (303.24,75.41) .. (305.55,75.87) .. controls (307.86,76.34) and (309.25,75.41) .. (309.71,73.1) .. controls (310.16,70.79) and (311.55,69.86) .. (313.86,70.32) .. controls (316.17,70.78) and (317.56,69.85) .. (318.02,67.54) .. controls (318.48,65.23) and (319.87,64.3) .. (322.18,64.77) .. controls (324.49,65.23) and (325.88,64.3) .. (326.34,61.99) .. controls (326.79,59.68) and (328.18,58.75) .. (330.49,59.21) .. controls (332.8,59.67) and (334.19,58.74) .. (334.65,56.43) .. controls (335.11,54.12) and (336.5,53.19) .. (338.81,53.66) -- (339.68,53.07) -- (346.34,48.63) ;
\draw [shift={(348,47.52)}, rotate = 146.26] [color={rgb, 255:red, 0; green, 0; blue, 0 }  ][line width=0.75]    (10.93,-3.29) .. controls (6.95,-1.4) and (3.31,-0.3) .. (0,0) .. controls (3.31,0.3) and (6.95,1.4) .. (10.93,3.29)   ;
\draw    (259.5,120.95) -- (447,120.52) ;
\draw [shift={(449,120.52)}, rotate = 179.87] [color={rgb, 255:red, 0; green, 0; blue, 0 }  ][line width=0.75]    (7.65,-2.3) .. controls (4.86,-0.97) and (2.31,-0.21) .. (0,0) .. controls (2.31,0.21) and (4.86,0.98) .. (7.65,2.3)   ;
\draw [shift={(259.5,120.95)}, rotate = 359.87] [color={rgb, 255:red, 0; green, 0; blue, 0 }  ][line width=0.75]      (0,-7.83) .. controls (-2.16,-7.83) and (-3.91,-6.07) .. (-3.91,-3.91) .. controls (-3.91,-1.75) and (-2.16,0) .. (0,0) ;
\draw    (466,113.51) .. controls (466.23,111.16) and (467.51,110.1) .. (469.86,110.33) .. controls (472.21,110.56) and (473.49,109.51) .. (473.72,107.16) .. controls (473.95,104.81) and (475.24,103.75) .. (477.59,103.98) .. controls (479.94,104.21) and (481.22,103.16) .. (481.45,100.81) .. controls (481.68,98.46) and (482.96,97.4) .. (485.31,97.63) .. controls (487.66,97.86) and (488.94,96.81) .. (489.17,94.46) .. controls (489.4,92.11) and (490.68,91.05) .. (493.03,91.28) .. controls (495.38,91.51) and (496.67,90.45) .. (496.9,88.1) .. controls (497.13,85.75) and (498.41,84.7) .. (500.76,84.93) .. controls (503.11,85.16) and (504.39,84.1) .. (504.62,81.75) .. controls (504.85,79.4) and (506.13,78.35) .. (508.48,78.58) .. controls (510.83,78.81) and (512.11,77.75) .. (512.34,75.4) .. controls (512.57,73.05) and (513.85,71.99) .. (516.2,72.22) .. controls (518.55,72.45) and (519.84,71.4) .. (520.07,69.05) .. controls (520.3,66.7) and (521.58,65.64) .. (523.93,65.87) .. controls (526.28,66.1) and (527.56,65.05) .. (527.79,62.7) .. controls (528.02,60.35) and (529.3,59.29) .. (531.65,59.52) .. controls (534,59.75) and (535.28,58.7) .. (535.51,56.35) -- (537.28,54.9) -- (543.46,49.82) ;
\draw [shift={(545,48.55)}, rotate = 140.57] [color={rgb, 255:red, 0; green, 0; blue, 0 }  ][line width=0.75]    (10.93,-3.29) .. controls (6.95,-1.4) and (3.31,-0.3) .. (0,0) .. controls (3.31,0.3) and (6.95,1.4) .. (10.93,3.29)   ;
\draw    (238.27,146.5) -- (239.73,286.86) ;
\draw [shift={(239.75,288.86)}, rotate = 269.4] [color={rgb, 255:red, 0; green, 0; blue, 0 }  ][line width=0.75]    (10.93,-3.29) .. controls (6.95,-1.4) and (3.31,-0.3) .. (0,0) .. controls (3.31,0.3) and (6.95,1.4) .. (10.93,3.29)   ;
\draw [shift={(238.27,146.5)}, rotate = 269.4] [color={rgb, 255:red, 0; green, 0; blue, 0 }  ][line width=0.75]    (10.93,-3.29) .. controls (6.95,-1.4) and (3.31,-0.3) .. (0,0) .. controls (3.31,0.3) and (6.95,1.4) .. (10.93,3.29)   ;
\draw    (257.87,300.78) -- (445.77,299.92) ;
\draw [shift={(447.77,299.91)}, rotate = 179.74] [color={rgb, 255:red, 0; green, 0; blue, 0 }  ][line width=0.75]    (7.65,-2.3) .. controls (4.86,-0.97) and (2.31,-0.21) .. (0,0) .. controls (2.31,0.21) and (4.86,0.98) .. (7.65,2.3)   ;
\draw [shift={(257.87,300.78)}, rotate = 359.74] [color={rgb, 255:red, 0; green, 0; blue, 0 }  ][line width=0.75]      (0,-7.83) .. controls (-2.16,-7.83) and (-3.91,-6.07) .. (-3.91,-3.91) .. controls (-3.91,-1.75) and (-2.16,0) .. (0,0) ;
\draw    (468.77,289.03) .. controls (468.73,286.67) and (469.89,285.47) .. (472.25,285.43) .. controls (474.61,285.4) and (475.77,284.2) .. (475.73,281.84) .. controls (475.69,279.48) and (476.85,278.28) .. (479.21,278.25) .. controls (481.57,278.22) and (482.73,277.02) .. (482.69,274.66) .. controls (482.65,272.31) and (483.81,271.11) .. (486.16,271.07) .. controls (488.52,271.04) and (489.68,269.84) .. (489.64,267.48) .. controls (489.6,265.12) and (490.76,263.92) .. (493.12,263.89) .. controls (495.48,263.86) and (496.64,262.66) .. (496.6,260.3) .. controls (496.56,257.94) and (497.72,256.74) .. (500.08,256.7) .. controls (502.44,256.67) and (503.6,255.47) .. (503.56,253.11) .. controls (503.52,250.75) and (504.68,249.55) .. (507.04,249.52) .. controls (509.4,249.49) and (510.56,248.29) .. (510.52,245.93) .. controls (510.48,243.57) and (511.64,242.37) .. (514,242.34) .. controls (516.36,242.31) and (517.52,241.11) .. (517.48,238.75) .. controls (517.44,236.39) and (518.6,235.19) .. (520.96,235.16) .. controls (523.31,235.12) and (524.47,233.92) .. (524.43,231.57) .. controls (524.39,229.21) and (525.55,228.01) .. (527.91,227.98) .. controls (530.27,227.94) and (531.43,226.74) .. (531.39,224.38) .. controls (531.35,222.02) and (532.51,220.82) .. (534.87,220.79) -- (537.4,218.18) -- (542.97,212.44) ;
\draw [shift={(544.36,211)}, rotate = 134.09] [color={rgb, 255:red, 0; green, 0; blue, 0 }  ][line width=0.75]    (10.93,-3.29) .. controls (6.95,-1.4) and (3.31,-0.3) .. (0,0) .. controls (3.31,0.3) and (6.95,1.4) .. (10.93,3.29)   ;
\draw    (249.87,292.46) .. controls (250.08,290.11) and (251.35,289.04) .. (253.7,289.24) .. controls (256.05,289.45) and (257.32,288.38) .. (257.53,286.03) .. controls (257.74,283.68) and (259.01,282.61) .. (261.36,282.81) .. controls (263.71,283.01) and (264.98,281.94) .. (265.18,279.59) .. controls (265.39,277.24) and (266.66,276.17) .. (269.01,276.38) .. controls (271.36,276.58) and (272.63,275.51) .. (272.84,273.16) .. controls (273.05,270.81) and (274.32,269.74) .. (276.67,269.94) .. controls (279.02,270.15) and (280.29,269.08) .. (280.5,266.73) .. controls (280.7,264.38) and (281.97,263.31) .. (284.32,263.51) .. controls (286.67,263.71) and (287.94,262.64) .. (288.15,260.29) .. controls (288.36,257.94) and (289.63,256.87) .. (291.98,257.08) .. controls (294.33,257.28) and (295.6,256.21) .. (295.81,253.86) .. controls (296.02,251.51) and (297.29,250.44) .. (299.64,250.64) .. controls (301.99,250.85) and (303.26,249.78) .. (303.46,247.43) .. controls (303.67,245.08) and (304.94,244.01) .. (307.29,244.21) .. controls (309.64,244.41) and (310.91,243.34) .. (311.12,240.99) .. controls (311.33,238.64) and (312.6,237.57) .. (314.95,237.78) .. controls (317.3,237.98) and (318.57,236.91) .. (318.78,234.56) .. controls (318.98,232.21) and (320.25,231.14) .. (322.6,231.34) .. controls (324.95,231.55) and (326.22,230.48) .. (326.43,228.13) .. controls (326.64,225.78) and (327.91,224.71) .. (330.26,224.91) .. controls (332.61,225.11) and (333.88,224.04) .. (334.09,221.69) .. controls (334.3,219.34) and (335.57,218.27) .. (337.92,218.48) -- (341.34,215.6) -- (347.47,210.45) ;
\draw [shift={(349,209.16)}, rotate = 139.96] [color={rgb, 255:red, 0; green, 0; blue, 0 }  ][line width=0.75]    (10.93,-3.29) .. controls (6.95,-1.4) and (3.31,-0.3) .. (0,0) .. controls (3.31,0.3) and (6.95,1.4) .. (10.93,3.29)   ;
\draw    (555.92,64) -- (555.54,186) ;
\draw [shift={(555.54,188)}, rotate = 270.18] [color={rgb, 255:red, 0; green, 0; blue, 0 }  ][line width=0.75]    (10.93,-3.29) .. controls (6.95,-1.4) and (3.31,-0.3) .. (0,0) .. controls (3.31,0.3) and (6.95,1.4) .. (10.93,3.29)   ;
\draw [shift={(555.92,64)}, rotate = 270.18] [color={rgb, 255:red, 0; green, 0; blue, 0 }  ][line width=0.75]    (10.93,-3.29) .. controls (6.95,-1.4) and (3.31,-0.3) .. (0,0) .. controls (3.31,0.3) and (6.95,1.4) .. (10.93,3.29)   ;
\draw    (457.61,145) -- (458.21,285.86) ;
\draw [shift={(458.22,287.86)}, rotate = 269.75] [color={rgb, 255:red, 0; green, 0; blue, 0 }  ][line width=0.75]    (10.93,-3.29) .. controls (6.95,-1.4) and (3.31,-0.3) .. (0,0) .. controls (3.31,0.3) and (6.95,1.4) .. (10.93,3.29)   ;
\draw [shift={(457.61,145)}, rotate = 269.75] [color={rgb, 255:red, 0; green, 0; blue, 0 }  ][line width=0.75]    (10.93,-3.29) .. controls (6.95,-1.4) and (3.31,-0.3) .. (0,0) .. controls (3.31,0.3) and (6.95,1.4) .. (10.93,3.29)   ;
\end{tikzpicture}
\end{adjustbox}

Note that the cubes compose in each of the three directions. 
The horizontal and vertical stacking of cubes 
is given by composing the SK-embeddings of cobordisms  fiberwise.
In the cobordism direction, the  fiberwise composition of two cobordisms  bundles
with corresponding 
 manifold bundles on the boundary
is defined to be a cobordism bundle over $\Delta^r$ where the maps $l$ and $l'$ are added fiberwise and the cobordisms are fiberwise composed along the common boundary over a point in $\Delta^r$.

\subsection{General construction of  $X^*_{\bullet, \bullet,\bullet}$}
 Now we use elements of $X^r_{a,b,c}$ with $a,b,c \in \{0,1\}$ as building blocks to define the simplices of arbitrary size. Any admissible composition of the building blocks over $\Delta^r$ that form a diagram of the size $k \times l \times m$ with $k,l,m \in \mathbb{Z}_{\geq 0}$ represents an element in $X^r_{k,l,m}$. Note that for $k,l,m \geq 1$ the set $X^r_{k,l,m}$ of all $(k,l,m)$-simplices over $\Delta^r$ is given by a set of all ``cubical'' diagrams of size $k \times l \times m$, consisting of $klm$ copies of size 1 cubes fibered 
over $\Delta^r$.

\section{Recovering cobordism cut and paste groups as $K^\cube_0$} \label{pi_0ComputationSec}

In this section we build a spectrum out of the  quadrisimplicial set $X^*_{\bullet, \bullet, \bullet}$
and show it recovers the cobordism cut and paste group with boundary $\SKbar^\d_n$ on $\pi_0$.
The proof of the latter (\autoref{thm:pi_0}), is a generalisation of \cite[Theorem 3.1]{campbell2023algebraic}, but requires more complicated simplicial combinatorics to deal with the additional simplicial directions.
We define $X_\bullet$ to be the diagonal of the quadrisimplicial set. 
Note that for a quadrisimplicial set, the realisation of the diagonal simplicial set is homotopy equivalent to the realisation of the quadrisimplicial set itself
$$|X_\bullet| \simeq |X^*_{\bullet, \bullet, \bullet}|.$$



\begin{remarksection}\upshape
Note that $X_\bullet$ cannot be interpreted as the nerve of 
a topological category as there is no natural way to define a composition map $$X_1 \underset{X_0}\times X_1 \rightarrow X_2.$$ Given two cubes  over $\Delta^1$ that agree on the appropriate vertex, it is not always possible to find a $2\times 2\times 2$ cube  over $\Delta^2$ incorporating the two. This also holds on the level of squares of SK-embeddings in  $X^*_{\bullet, \bullet, 0}$. For an example, see Figure \ref{squares-do-not-extend}.
\end{remarksection}

\begin{figure}[ht]

\adjustbox{scale=0.40}{

\tikzset{every picture/.style={line width=0.75pt}} 

\begin{tikzpicture}[x=0.75pt,y=0.75pt,yscale=-1,xscale=1]

\draw (309.5,467.49) node  {\includegraphics[width=466.5pt,height=458.98pt]{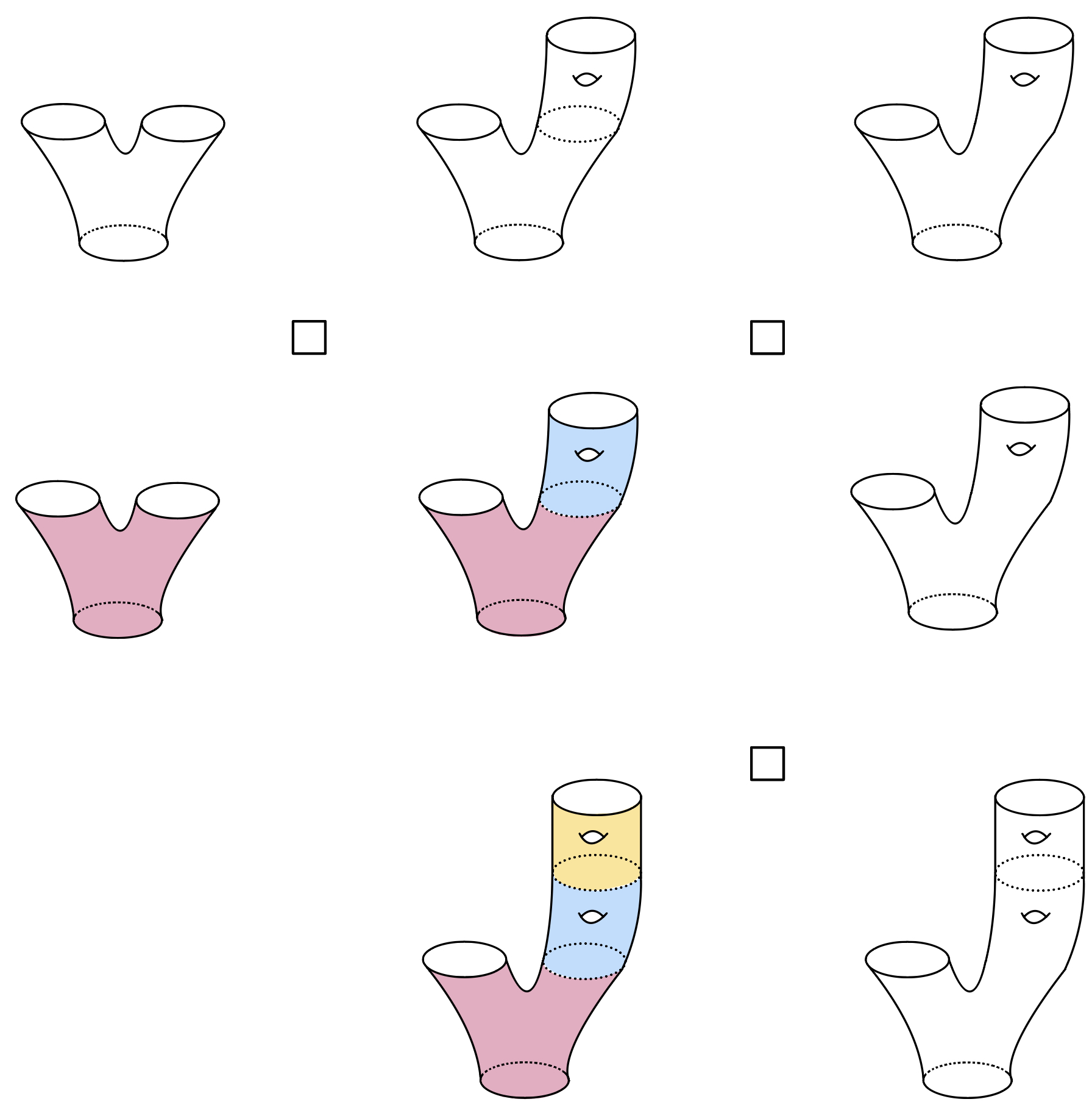}};
\draw    (150.5,269.5) -- (172,269.5) -- (210.5,269.5) ;
\draw [shift={(212.5,269.5)}, rotate = 180] [color={rgb, 255:red, 0; green, 0; blue, 0 }  ][line width=0.75]    (10.93,-3.29) .. controls (6.95,-1.4) and (3.31,-0.3) .. (0,0) .. controls (3.31,0.3) and (6.95,1.4) .. (10.93,3.29)   ;
\draw [shift={(150.5,269.5)}, rotate = 0] [color={rgb, 255:red, 0; green, 0; blue, 0 }  ][line width=0.75]      (0,-11.18) .. controls (-3.09,-11.18) and (-5.59,-8.68) .. (-5.59,-5.59) .. controls (-5.59,-2.5) and (-3.09,0) .. (0,0) ;
\draw    (396.04,271.15) -- (370.04,271.15) -- (370.04,268.15) -- (471.51,268.5)(396.04,268.15) -- (370.04,268.15) -- (370.03,271.15) -- (471.49,271.5) ;
\draw    (398.04,467.15) -- (372.04,467.15) -- (372.04,464.15) -- (473.51,464.5)(398.04,464.15) -- (372.04,464.15) -- (372.03,467.15) -- (473.49,467.5) ;
\draw    (405.04,729.15) -- (379.04,729.15) -- (379.04,726.15) -- (480.51,726.5)(405.04,726.15) -- (379.04,726.15) -- (379.03,729.15) -- (480.49,729.5) ;
\draw    (149.5,468.5) -- (171,468.5) -- (209.5,468.5) ;
\draw [shift={(211.5,468.5)}, rotate = 180] [color={rgb, 255:red, 0; green, 0; blue, 0 }  ][line width=0.75]    (10.93,-3.29) .. controls (6.95,-1.4) and (3.31,-0.3) .. (0,0) .. controls (3.31,0.3) and (6.95,1.4) .. (10.93,3.29)   ;
\draw [shift={(149.5,468.5)}, rotate = 0] [color={rgb, 255:red, 0; green, 0; blue, 0 }  ][line width=0.75]      (0,-11.18) .. controls (-3.09,-11.18) and (-5.59,-8.68) .. (-5.59,-5.59) .. controls (-5.59,-2.5) and (-3.09,0) .. (0,0) ;
\draw  [dash pattern={on 0.84pt off 2.51pt}]  (105.5,728.5) -- (183,728.5) -- (221.5,728.5) ;
\draw [shift={(223.5,728.5)}, rotate = 180] [color={rgb, 255:red, 0; green, 0; blue, 0 }  ][line width=0.75]    (10.93,-3.29) .. controls (6.95,-1.4) and (3.31,-0.3) .. (0,0) .. controls (3.31,0.3) and (6.95,1.4) .. (10.93,3.29)   ;
\draw    (297.7,324.75) -- (297.7,369.27)(294.7,324.75) -- (294.7,369.27) ;
\draw    (546.7,320.75) -- (546.7,365.27)(543.7,320.75) -- (543.7,365.27) ;
\draw    (298.01,536.2) -- (298.79,575.77) ;
\draw [shift={(298.83,577.77)}, rotate = 268.87] [color={rgb, 255:red, 0; green, 0; blue, 0 }  ][line width=0.75]    (10.93,-3.29) .. controls (6.95,-1.4) and (3.31,-0.3) .. (0,0) .. controls (3.31,0.3) and (6.95,1.4) .. (10.93,3.29)   ;
\draw [shift={(298.01,536.2)}, rotate = 268.87] [color={rgb, 255:red, 0; green, 0; blue, 0 }  ][line width=0.75]    (10.93,-3.29) .. controls (6.95,-1.4) and (3.31,-0.3) .. (0,0) .. controls (3.31,0.3) and (6.95,1.4) .. (10.93,3.29)   ;
\draw    (544.01,536.2) -- (544.79,575.77) ;
\draw [shift={(544.83,577.77)}, rotate = 268.87] [color={rgb, 255:red, 0; green, 0; blue, 0 }  ][line width=0.75]    (10.93,-3.29) .. controls (6.95,-1.4) and (3.31,-0.3) .. (0,0) .. controls (3.31,0.3) and (6.95,1.4) .. (10.93,3.29)   ;
\draw [shift={(544.01,536.2)}, rotate = 268.87] [color={rgb, 255:red, 0; green, 0; blue, 0 }  ][line width=0.75]    (10.93,-3.29) .. controls (6.95,-1.4) and (3.31,-0.3) .. (0,0) .. controls (3.31,0.3) and (6.95,1.4) .. (10.93,3.29)   ;
\draw    (69.7,328.75) -- (69.7,373.27)(66.7,328.75) -- (66.7,373.27) ;
\draw  [dash pattern={on 0.84pt off 2.51pt}]  (67.97,547.2) -- (68.82,683.77) ;
\draw [shift={(68.83,685.77)}, rotate = 269.64] [color={rgb, 255:red, 0; green, 0; blue, 0 }  ][line width=0.75]    (10.93,-3.29) .. controls (6.95,-1.4) and (3.31,-0.3) .. (0,0) .. controls (3.31,0.3) and (6.95,1.4) .. (10.93,3.29)   ;

\draw (62.25,710.99) node [anchor=north west][inner sep=0.75pt]  [font=\footnotesize] [align=left] {{\huge {\fontfamily{ptm}\selectfont ?}}};
\draw (145.25,564.99) node [anchor=north west][inner sep=0.75pt]  [font=\footnotesize] [align=left] {\begin{minipage}[lt]{41.24pt}\setlength\topsep{0pt}
\begin{center}
{\Large {\fontfamily{ptm}\selectfont Not}}\\{\Large {\fontfamily{ptm}\selectfont possible}}
\end{center}

\end{minipage}};

\end{tikzpicture}

}
\caption{Composing squares (or cubes) that coincide in only one vertex is not always possible}
\label{squares-do-not-extend}
\end{figure}
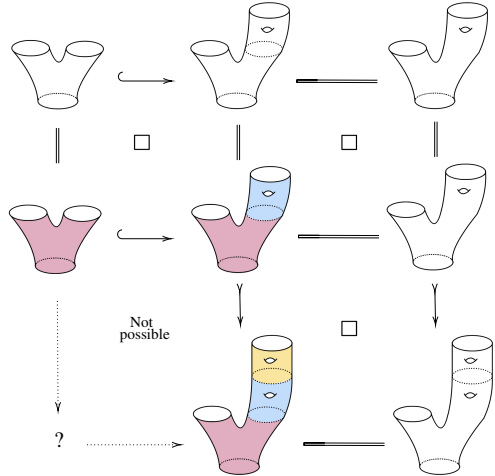


\begin{definitionsection}\upshape
 We define $$ K^\cube(\mbar) \coloneqq \Omega_{\varnothing} | X^*_{\bullet, \bullet, \bullet} |  \simeq \Omega_{\varnothing} | X_{\bullet} | .$$
\end{definitionsection}


In the proofs below we will use a bracket to denote 1-simplices (cubes  over $\Delta^1$) and 2-simplices ($2\times 2 \times 2$ cubes  over $\Delta^2$) that are generated from the diagram inside brackets by extending with degeneracies in all other directions, for example:

   	$\left[\begin{tikzcd}[scale cd=0.7, row sep=scriptsize,column sep=scriptsize]
 A \arrow[r, hook] & B 
\end{tikzcd} \right] \coloneqq \begin{tikzcd}[scale cd=0.7,row sep=small,column sep=small]& 
    A \arrow[rr, hook]\arrow[equal, dd] & & B\arrow[equal, dd] \\
    A\arrow[rr,crossing over, hook] \arrow[ur, equal] \arrow[dd, equal] & &  B \arrow[ur, equal]\\
    	& A \arrow[rr, hook] &  & 
    	B\\
    A\arrow[rr, hook] \arrow[ur, equal] & & B\arrow[from=uu,crossing over, equal] \arrow[ur, equal]\\
  	\end{tikzcd} ~~ ~ $
and
$ ~~~
\left[\begin{tikzcd}[scale cd=0.6, row sep=small,column sep=small] 
&& A' \arrow[dd, tail] \\
A \arrow[urr, squiggly] \arrow[dd, tail] &&\\
&&C'\\
C \arrow[urr, squiggly] &&
\end{tikzcd} \right]\coloneqq \begin{tikzcd}[scale cd=0.7,row sep=small,column sep=small]& 
    A' \arrow[rr, equal]\arrow[tail, dd] & & A'\arrow[tail, dd] \\
    A\arrow[rr,crossing over, equal] \arrow[ur, squiggly] \arrow[dd, tail] & &  A \arrow[ur, squiggly]\\
    	& C' \arrow[rr, equal] &  & 
    	C'\\
    C\arrow[rr, equal] \arrow[ur, squiggly] & & C\arrow[from=uu,crossing over, tail] \arrow[ur, squiggly]\\
  	\end{tikzcd}, 
$
 where we think of vertices of the cubes on the right as being the total spaces of fibre bundles over $\Delta^1$, which fibrewise has the shape of the prescribed cube.
Note that the diagonal face maps ensure that the source and target of a 1-simplex are single manifolds with boundary, fibred over $\Delta^0 = *$.


We now prove our main theorem, that recovers the cobordism cut and paste groups with boundary as $\pi_0$ of the cubes $K$-theory of manifolds spectrum. The strategy for this proof will be start by contracting all vertices to a single vertex in the diagonal of the quadrisimiplicial set $X^*_{\bullet, \bullet, \bullet}$, 
and subsequently compute $\pi_1$ (corresponding to $\pi_0$ of the spectrum) using direct combinatorial arguments. 

\begin{thm} \label{thm:pi_0}
There is an isomorphism
$$\pi_0 K^\cube(\mbar) \cong \SKbar^\partial_n.$$
\end{thm}

\begin{proof}
 Let $Y \coloneqq \vert X_{\bullet}\vert$.
We have $$\pi_0 K^\cube(\mbar) = \pi_1 | X^*_{\bullet, \bullet, \bullet}| \cong \pi_1 Y.$$

First we will contract all vertices in $Y$ to a single vertex. 
Consider the set of 1-simplices of the form $\left[ \begin{tikzcd}[scale cd=0.6, row sep=scriptsize,column sep=scriptsize] 
A \arrow[r, hook] & B
\end{tikzcd}\right]$. 
	For each cube of this shape, there is a canonical 2-simplex of the following shape 
\begin{equation}\left[ \begin{tikzcd}[column sep=tiny, row sep=tiny]
    \varnothing \arrow[rr, hook]\arrow[dd, equal] && A \arrow[rr, hook] \arrow[dd, equal] && B \arrow[dd, equal] \\ \\
    \varnothing \arrow[rr, hook]\arrow[dd, equal] && A \arrow[rr, hook] \arrow[dd, equal] && B \arrow[dd, equal] \\ \\
    \varnothing \arrow[rr, hook] && A \arrow[rr, hook] && B
    \end{tikzcd}\right],\label{eqn:contracting}\end{equation}
where we extend twice in the cobordism direction using identity cobordisms and  the induced $\SK$-embeddings on the cobordisms. This 2-simplex has boundary
$$
\left[ \begin{tikzcd}[scale cd=0.6, row sep=scriptsize,column sep=scriptsize]
\varnothing \arrow[r, hook] & A \end{tikzcd}\right]
+ \left[ \begin{tikzcd}[scale cd=0.6, row sep=scriptsize,column sep=scriptsize]
 A \arrow[r, hook] & B \end{tikzcd}\right]
-\left[ \begin{tikzcd}[scale cd=0.6, row sep=scriptsize,column sep=scriptsize]
\varnothing \arrow[r, hook] & B \end{tikzcd}\right].
$$
 The subset 
 generated by all 0-simplices together with the 1- and 2-simplices above 
 form a subsimplicial set $X'_\bullet$ of $X_\bullet$, of which the realisation $Y'$ is homeomorphic to a cone over all  non-degenerate 1-simplices of the shape $\left[\begin{tikzcd}[scale cd=0.7, row sep=scriptsize,column sep=scriptsize]
 A \arrow[r, hook] & B \end{tikzcd} \right]$
  for $A \neq \varnothing$, with conepoint $\varnothing$.
As a cone is contractible and $Y$ and $Y'$ form a CW pair, we can perform a topological quotient in $Y$ leaving the homotopy type unchanged.
The realisation of a quotient of simplicial sets is the quotient of the realisations, so we have that $Y \simeq Y/Y' \cong |X_\bullet/X'_\bullet|$.  Compare also with the proof of \cite[Theorem 3.1]{campbell2023algebraic}.



 
We will now compute  $\pi_1 Y \cong \pi_1 |X_\bullet/X'_\bullet|$. Note that the fundamental group of a simplicial set  with a single vertex  has a presentation with generators given by all 1-simplices and relations given by 2-simplices, where a 2-simplex 
\[\begin{tikzcd}[column sep=1.5em]
 & B \arrow{dr}{g} \\
A \arrow{ur}{f} \arrow{rr}{h} && C
\end{tikzcd}
\]
induces the relation $f+g-h =0$. In practice this means that a 2-simplex represented by a $2\times 2\times 2$-cube  over $\Delta^2$ induces the relation that the initial plus the final cubes  as $\Delta^1$-bundles equals the composite cube  over $\Delta^1$ viewed as 1-simplices.

First we will reduce the generators of $\pi_1 Y$ using 2-simplices.
    
  Consider a generic $1$-simplex (in red in the figure bellow).   Extending via identities,
    we obtain the following $2$-simplex:

\begin{adjustbox}{scale=0.5, center}

\tikzset{every picture/.style={line width=0.75pt}} 



    \end{adjustbox}
    
    This implies the relation  
   \begin{equation} \label{generic-cube-relation} \left[\begin{tikzcd}[scale cd=0.6, row sep=scriptsize,column sep=scriptsize]
    	A \arrow[d, tail] \\
    	C\end{tikzcd}\right]  + \left[\begin{tikzcd}[scale cd=0.6, row sep=scriptsize,column sep=scriptsize] & C' \arrow[rr, hook] && D' \\
    	C \arrow[ur, squiggly, "W_2"] \arrow[rr, hook] && D  \arrow[ur, squiggly, "W_3"]& \end{tikzcd} \right]  = \left[ \textnormal{generic $1$-simplex}\right].
	\end{equation}

    Analogously, we can construct the following $2$-simplex and deduce the relation in Equation \eqref{back-vertical}.
    
\begin{adjustbox}{scale=0.5, center}

\tikzset{every picture/.style={line width=0.75pt}} 



     \end{adjustbox}

   \begin{equation} \label{back-vertical}  \left[\begin{tikzcd}[scale cd=0.6, row sep=scriptsize,column sep=scriptsize] & A' \arrow[rr, hook] && B' \\
    	A \arrow[ur, squiggly, "W_0"] \arrow[rr, hook] && B  \arrow[ur, squiggly, "W_1"]& \end{tikzcd} \right] 
    	+ \left[\begin{tikzcd}[scale cd=0.6, row sep=scriptsize,column sep=scriptsize]
    	B' \arrow[d, tail] \\
    	D'\end{tikzcd}\right] 
    	= \left[ \textnormal{generic $1$-simplex}\right].
	\end{equation}

         Next we consider the following  $2$-simplex, which is obtained by extending the (non-generic) cube in green

\begin{adjustbox}{scale=0.5, center}

\tikzset{every picture/.style={line width=0.75pt}} 



\end{adjustbox}

        This establishes the relation        
    $$\left[\begin{tikzcd}[scale cd=0.6, row sep=scriptsize,column sep=scriptsize] & A' \\
    	A \arrow[ur, squiggly, "W_0"] & \end{tikzcd}\right]  +   \left[ \begin{tikzcd}[scale cd=0.6, row sep=scriptsize,column sep=scriptsize] 
A' \arrow[r, hook] & B'
\end{tikzcd}\right] =   \left[\begin{tikzcd}[scale cd=0.6, row sep=scriptsize,column sep=scriptsize] & A' \arrow[rr, hook] && B' \\
    	A \arrow[ur, squiggly, "W_0"] \arrow[rr, hook] && B  \arrow[ur, squiggly, "W_1"]& \end{tikzcd} \right] $$
    But $\left[ \begin{tikzcd}[scale cd=0.6, row sep=scriptsize,column sep=scriptsize] 
A' \arrow[r, hook] & B'
\end{tikzcd}\right]$ was contracted, therefore
\begin{equation}\label{cobW0}
\left[\begin{tikzcd}[scale cd=0.6, row sep=scriptsize,column sep=scriptsize] & A' \\
    	A \arrow[ur, squiggly, "W_0"] & \end{tikzcd} \right]   =   \left[\begin{tikzcd}[scale cd=0.6, row sep=scriptsize,column sep=scriptsize] & A' \arrow[rr, hook] && B' \\
    	A \arrow[ur, squiggly, "W_0"] \arrow[rr, hook] && B  \arrow[ur, squiggly, "W_1"]& \end{tikzcd} \right] 
\end{equation}

    Analogously, one can construct a  $2$-simplex to obtain the relation
    $$\left[ \begin{tikzcd}[scale cd=0.6, row sep=scriptsize,column sep=scriptsize] 
A \arrow[r, hook] & B
\end{tikzcd}\right]  +\left[\begin{tikzcd}[scale cd=0.6, row sep=scriptsize,column sep=scriptsize] & B' \\
    	B \arrow[ur, squiggly, "W_1"] & \end{tikzcd}\right]  =   \left[\begin{tikzcd}[scale cd=0.6, row sep=scriptsize,column sep=scriptsize] & A' \arrow[rr, hook] && B' \\
    	A \arrow[ur, squiggly, "W_0"] \arrow[rr, hook] && B  \arrow[ur, squiggly, "W_1"]& \end{tikzcd} \right] $$
   so that 
\begin{equation}\label{cobW1}
\left[\begin{tikzcd}[scale cd=0.6, row sep=scriptsize,column sep=scriptsize] & B' \\
    	B \arrow[ur, squiggly, "W_1"] & \end{tikzcd} \right]   =   \left[\begin{tikzcd}[scale cd=0.6, row sep=scriptsize,column sep=scriptsize] & A' \arrow[rr, hook] && B' \\
    	A \arrow[ur, squiggly, "W_0"] \arrow[rr, hook] && B  \arrow[ur, squiggly, "W_1"]& \end{tikzcd} \right] 
\end{equation}
From Equations \eqref{cobW0} and \eqref{cobW1} we see that cobordisms that fit into a square are in the same class.

Note that we are leaving out the length of the cobordism in our notation for the following reason. 
 Cobordisms either have length constant 0 over $\Delta^1$, or they have varying non-vanishing length $l:\Delta^1 \rightarrow (0,\infty)$.
The $1$-simplices  
$$\left[\begin{tikzcd}[scale cd=0.6, row sep=scriptsize,column sep=scriptsize] & A \\
    	A \arrow[ur, squiggly, "Id"] & \end{tikzcd},  0 \right]$$
with cobordisms of length $0$ 
are degenerate and these do not contribute any non-trivial loops. 
    Now note that any cobordism $W$ fits into a cube of the shape
$$
\left[\begin{tikzcd}[scale cd=0.6, row sep=scriptsize,column sep=scriptsize] & \varnothing \arrow[rr, hook] && B' \\
    	\varnothing \arrow[ur, squiggly, "\varnothing"] \arrow[rr, hook] && B'  \arrow[ur, squiggly, "W"]& \end{tikzcd} ,l \right] 
$$
This implies that 
\begin{equation}\label{cob-empty}
\left[\begin{tikzcd}[scale cd=0.6, row sep=scriptsize,column sep=scriptsize] & B' \\
    	B \arrow[ur, squiggly, "W"] & \end{tikzcd}, l \right]  = \left[\begin{tikzcd}[scale cd=0.6, row sep=scriptsize,column sep=scriptsize] & \varnothing \\
    	\varnothing \arrow[ur, squiggly, "\varnothing"] & \end{tikzcd} , l\right]  
\end{equation}

 Consider any 2-simplex with three sides of the form
$[l]:= \left[\begin{tikzcd}[scale cd=0.6, row sep=scriptsize,column sep=scriptsize] & \varnothing \\
    	\varnothing \arrow[ur, squiggly, "\varnothing"] & \end{tikzcd} , l\right]$,  where the three functions $l$ are given by
     \[\begin{tikzcd}[column sep=1.5em]
 & \varnothing \arrow{dr}{l_3} & &  \\
\varnothing \arrow{ur}{l_1} \arrow{rr}{l_2} & & \varnothing
\end{tikzcd}.
\]
 Such a 2-simplex exists for non-degenerate sides whenever $l_3(0) > l_1(0)$ and $l_3(1)>l_2(1)$. For instance, we have let $l_1$ be a generic function $l: \Delta^1 \to (0,\infty)$, $l_2 = l(0)+1:= c$ and $l_3 = 1$, which imposes the relation $$[l] = [c] + [1].$$
Now consider $l_1= A$, $ l_2=B$ and $l_3 = C$ with constants $C>A, B$. This imposes the relation $$[C] = [A]+[B].$$
Taking $A=B=1$ and $C>1$ we get $[C]=2[1]$, thus in particular $[l] = 3[1]$.
We also have $[3]=[2]+[1]$ and $[3]=[2]=2 [1]$, so we conclude that $[1]=0$ and hence $[l]=0$.



    From  the above and Equation \eqref{cob-empty} we deduce that there are no additional generators given by cobordism, since the class of an empty cobordism (of non-zero length) is zero.

    Therefore, from Equations \eqref{generic-cube-relation}  and   \eqref{cob-empty}, we see that $\pi_1 Y$ is generated by cubes of the form $ \left[\begin{tikzcd}[scale cd=0.6, row sep=scriptsize,column sep=scriptsize]
    	A \arrow[d, tail] \\
    	C\end{tikzcd}\right] $. 
Since we established that cubes of the shape  $$\left[\begin{tikzcd}[scale cd=0.6, row sep=scriptsize,column sep=scriptsize] & A' \arrow[rr, hook] && B' \\
    	A \arrow[ur, squiggly, "W_0"] \arrow[rr, hook] && B  \arrow[ur, squiggly, "W_1"]& \end{tikzcd} \right]$$ represent trivial loops, from \eqref{generic-cube-relation} and \eqref{back-vertical} we get that
\begin{equation} \label{vertical arrows are the same}
   \left[  \text{generic 1-simplex}   \right]  = \left[\begin{tikzcd}[scale cd=0.6, row sep=scriptsize,column sep=scriptsize]
    	A \arrow[d, tail] \\
    	C\end{tikzcd}\right]  =  \left[\begin{tikzcd}[scale cd=0.6, row sep=scriptsize,column sep=scriptsize]
    	B' \arrow[d, tail] \\
    	D' \end{tikzcd}\right].
\end{equation}

  We can deduce from the following relations of cubes
    \begin{equation}
    \partial \left[ \begin{tikzcd}[column sep=tiny, row sep=tiny]
    A \arrow[rr, equal]\arrow[dd, tail] && A \arrow[rr, hook] \arrow[dd, tail] && B \arrow[dd, tail] \\ \\
    C \arrow[rr, equal]\arrow[dd, equal] && C \arrow[rr, hook] \arrow[dd, equal] && D \arrow[dd, equal] \\ \\
    C \arrow[rr, equal] && C \arrow[rr, hook] && D
    \end{tikzcd} \right]   =    \left[\begin{tikzcd}[scale cd=0.6, row sep=scriptsize,column sep=scriptsize]
    	A \arrow[d, tail] \\
    	C\end{tikzcd}\right] + \left[ \begin{tikzcd}[scale cd=0.6, row sep=scriptsize,column sep=scriptsize] 
C \arrow[r, hook] & D
\end{tikzcd}\right] -  \left[  \begin{tikzcd}[column sep=tiny, row sep=tiny]   A \arrow[rr, hook] \arrow[dd, tail] &&  B \arrow[dd, tail] \\ 
  \\  C \arrow[rr, hook] && D \end{tikzcd}   \right] =0
  \label{eqn:firstone}
\end{equation}

 \begin{equation}
   \partial \left[ \begin{tikzcd}[column sep=tiny, row sep=tiny]
    A \arrow[rr, hook]\arrow[dd, equal] && B \arrow[rr, equal] \arrow[dd, equal] && B \arrow[dd, equal] \\ \\
    A \arrow[rr, hook]\arrow[dd, tail] && B \arrow[rr, equal] \arrow[dd, tail] && B \arrow[dd, tail] \\ \\
    C \arrow[rr, hook] && D \arrow[rr, equal] && D
    \end{tikzcd} \right]   =   \left[ \begin{tikzcd}[scale cd=0.6, row sep=scriptsize,column sep=scriptsize] 
A \arrow[r, hook] & B
\end{tikzcd}\right]   +  \left[\begin{tikzcd}[scale cd=0.6, row sep=scriptsize,column sep=scriptsize]
    	B \arrow[d, tail] \\
    	D\end{tikzcd}\right] - \left[  \begin{tikzcd}[column sep=tiny, row sep=tiny]   A \arrow[rr, hook] \arrow[dd, tail] &&  B \arrow[dd, tail] \\ 
  \\  C \arrow[rr, hook] && D \end{tikzcd}   \right]  =0
\end{equation}
that
  \begin{equation}
   \left[  \begin{tikzcd}[column sep=tiny, row sep=tiny]   A \arrow[rr, hook] \arrow[dd, tail] &&  B \arrow[dd, tail] \\ 
  \\  C \arrow[rr, hook] && D \end{tikzcd}   \right]  = \left[\begin{tikzcd}[scale cd=0.6, row sep=scriptsize,column sep=scriptsize]
    	A \arrow[d, tail] \\
    	C\end{tikzcd}\right]  =  \left[\begin{tikzcd}[scale cd=0.6, row sep=scriptsize,column sep=scriptsize]
    	B \arrow[d, tail] \\
    	D\end{tikzcd}\right]. \end{equation}
Note that the composition of embeddings $\varnothing \to A$ and $A \to C$ is an embedding $\varnothing \to C$, and the 2-simplex corresponding to this composition gives the relation

	 \begin{equation}
	 \left[\begin{tikzcd}[scale cd=0.6, row sep=scriptsize,column sep=scriptsize]
    	A \arrow[d, tail] \\
    	C\end{tikzcd}\right]   =  \left[\begin{tikzcd}[scale cd=0.6, row sep=scriptsize,column sep=scriptsize]
    	\varnothing \arrow[d, tail] \\
    	C\end{tikzcd}\right]  -   \left[\begin{tikzcd}[scale cd=0.6, row sep=scriptsize,column sep=scriptsize]
    	\varnothing \arrow[d, tail] \\
    	A \end{tikzcd}\right]  \\
	\eqqcolon \, [C] - [A].
\end{equation}
	So we deduce that $\pi_1Y$ can be generated by 1-simplices of the form $\left[\begin{tikzcd}[scale cd=0.6, row sep=scriptsize,column sep=scriptsize]
    	\varnothing \arrow[d, tail] \\
    	A \end{tikzcd}\right]$ 
     that we denote by $[A]$ for $A$  any manifold bundle over $\Delta^1$. 

	Moreover, whenever $A,B,C,D$ fit into a square  $$\left[  \begin{tikzcd}[column sep=tiny, row sep=tiny]   A \arrow[rr, hook] \arrow[dd, tail] &&  B \arrow[dd, tail] \\ 
  \\  C \arrow[rr, hook] && D \end{tikzcd}   \right]$$ we obtain the square relation 
	\begin{equation}\label{square relation}
	[C] - [A]=[D] - [B].
	\end{equation}

 In particular, we have the square

$$\left[  \begin{tikzcd}[column sep=tiny, row sep=tiny]   \varnothing \arrow[rr, hook, "="] \arrow[dd, tail] &&  \varnothing \arrow[dd, tail] \\ 
  \\  A \arrow[rr, hook, "\cong"] && A_0 \times [0,1]\end{tikzcd}   \right],$$
 where $A_0 \times [0,1]$ denotes a bundle over $\Delta^1$ with fiber $A_0$ that is a trivialisation of the bundle $A$. This shows that any manifold bundle representative can be reduced to a product bundle, which we denote by its fiber, i.e. $[A_0]$:
\begin{equation}\label{eqn:mfdbundle}
    [A] = [A_0]. 
\end{equation}

We have the pushout square of manifolds 
$$
   \left[  \begin{tikzcd}[column sep=tiny, row sep=tiny]   \varnothing \arrow[rr, hook] \arrow[dd, tail] &&  M \arrow[dd, tail] \\ 
  \\  N \arrow[rr, hook] && M \sqcup N \end{tikzcd}   \right]  $$
from which we conclude that 
\begin{equation}
     [N] = [M \sqcup N] - [M],
\end{equation}
hence the operation in $\pi_1 Y$ corresponds to disjoint union of manifolds.
	
So far we have established that $\pi_1 Y$ is generated by classes represented by manifolds adhering to the square relation (\ref{square relation}). We now establish the full set of relations on the generators.

	Consider a generic $2$-simplex with vertices labeled as $A_{ijk}$, where $0 \leq i, j, k \leq 2$ (see figure below).
From the arguments above, we know that a generic cube can be represented by a cube with one non-trivial vertical arrow and identities in horizontal directions, where any of the four vertical arrows can be chosen as a representative.
	This means that the relation arising from a generic 2-simplex can be expressed as Equation \eqref{relation} as illustrated in the diagram below.

\begin{adjustbox}{scale=0.45, center}
\tikzset{every picture/.style={line width=0.75pt}} 

\tikzset{every picture/.style={line width=0.75pt}} 

\begin{tikzpicture}[x=0.75pt,y=0.75pt,yscale=-1,xscale=1]

\draw  [color={rgb, 255:red, 0; green, 0; blue, 0 }  ,draw opacity=1 ][fill={rgb, 255:red, 0; green, 0; blue, 0 }  ,fill opacity=0 ] (116.5,160.71) -- (366.57,160.71) -- (366.57,416.5) -- (116.5,416.5) -- cycle ;
\draw [color={rgb, 255:red, 0; green, 0; blue, 0 }  ,draw opacity=1 ][fill={rgb, 255:red, 0; green, 0; blue, 0 }  ,fill opacity=0 ]   (116.5,160.71) -- (279.04,83.98) ;
\draw [color={rgb, 255:red, 0; green, 0; blue, 0 }  ,draw opacity=1 ][fill={rgb, 255:red, 0; green, 0; blue, 0 }  ,fill opacity=0 ]   (366.57,160.71) -- (447.84,122.34) ;
\draw [color={rgb, 255:red, 0; green, 0; blue, 0 }  ,draw opacity=1 ][fill={rgb, 255:red, 0; green, 0; blue, 0 }  ,fill opacity=0 ]   (447.84,122.34) -- (529.11,83.98) ;
\draw [color={rgb, 255:red, 0; green, 0; blue, 0 }  ,draw opacity=1 ][fill={rgb, 255:red, 0; green, 0; blue, 0 }  ,fill opacity=0 ]   (241.53,160.71) -- (322.8,122.34) ;
\draw [color={rgb, 255:red, 0; green, 0; blue, 0 }  ,draw opacity=1 ][fill={rgb, 255:red, 0; green, 0; blue, 0 }  ,fill opacity=0 ]   (322.8,122.34) -- (404.08,83.98) ;
\draw [color={rgb, 255:red, 0; green, 0; blue, 0 }  ,draw opacity=1 ][fill={rgb, 255:red, 0; green, 0; blue, 0 }  ,fill opacity=0 ]   (366.57,288.61) -- (447.84,250.24) ;
\draw [color={rgb, 255:red, 0; green, 0; blue, 0 }  ,draw opacity=1 ][fill={rgb, 255:red, 0; green, 0; blue, 0 }  ,fill opacity=0 ]   (447.84,250.24) -- (529.11,211.87) ;
\draw [color={rgb, 255:red, 0; green, 0; blue, 0 }  ,draw opacity=1 ][fill={rgb, 255:red, 0; green, 0; blue, 0 }  ,fill opacity=0 ] [dash pattern={on 0.84pt off 2.51pt}]  (241.53,288.61) -- (322.8,250.24) ;
\draw [color={rgb, 255:red, 0; green, 0; blue, 0 }  ,draw opacity=1 ][fill={rgb, 255:red, 0; green, 0; blue, 0 }  ,fill opacity=0 ] [dash pattern={on 0.84pt off 2.51pt}]  (322.8,250.24) -- (404.08,211.87) ;
\draw [color={rgb, 255:red, 0; green, 0; blue, 0 }  ,draw opacity=1 ][fill={rgb, 255:red, 0; green, 0; blue, 0 }  ,fill opacity=0 ]   (241.53,160.71) -- (241.53,416.5) ;
\draw [color={rgb, 255:red, 0; green, 0; blue, 0 }  ,draw opacity=1 ][fill={rgb, 255:red, 0; green, 0; blue, 0 }  ,fill opacity=0 ] [dash pattern={on 0.84pt off 2.51pt}]  (197.77,122.34) -- (197.77,378.13) ;
\draw [color={rgb, 255:red, 0; green, 0; blue, 0 }  ,draw opacity=1 ][fill={rgb, 255:red, 0; green, 0; blue, 0 }  ,fill opacity=0 ] [dash pattern={on 0.84pt off 2.51pt}]  (404.08,83.98) -- (404.08,339.76) ;
\draw [color={rgb, 255:red, 0; green, 0; blue, 0 }  ,draw opacity=1 ][fill={rgb, 255:red, 0; green, 0; blue, 0 }  ,fill opacity=0 ][line width=0.75]  [dash pattern={on 0.84pt off 2.51pt}]  (279.04,83.98) -- (279.04,193.96) -- (279.04,339.76) ;
\draw [color={rgb, 255:red, 0; green, 0; blue, 0 }  ,draw opacity=1 ][fill={rgb, 255:red, 0; green, 0; blue, 0 }  ,fill opacity=0 ] [dash pattern={on 0.84pt off 2.51pt}]  (197.77,378.13) -- (447.84,378.13) ;
\draw [color={rgb, 255:red, 0; green, 0; blue, 0 }  ,draw opacity=1 ][fill={rgb, 255:red, 0; green, 0; blue, 0 }  ,fill opacity=0 ] [dash pattern={on 0.84pt off 2.51pt}]  (197.77,250.24) -- (447.84,250.24) ;
\draw [color={rgb, 255:red, 0; green, 0; blue, 0 }  ,draw opacity=1 ][fill={rgb, 255:red, 0; green, 0; blue, 0 }  ,fill opacity=0 ]   (279.04,83.98) -- (529.11,83.98) ;
\draw [color={rgb, 255:red, 0; green, 0; blue, 0 }  ,draw opacity=1 ][fill={rgb, 255:red, 0; green, 0; blue, 0 }  ,fill opacity=0 ] [dash pattern={on 0.84pt off 2.51pt}]  (279.04,339.76) -- (529.11,339.76) ;
\draw [color={rgb, 255:red, 0; green, 0; blue, 0 }  ,draw opacity=1 ][fill={rgb, 255:red, 0; green, 0; blue, 0 }  ,fill opacity=0 ] [dash pattern={on 0.84pt off 2.51pt}]  (116.5,416.5) -- (279.04,339.76) ;
\draw [color={rgb, 255:red, 0; green, 0; blue, 0 }  ,draw opacity=1 ][fill={rgb, 255:red, 0; green, 0; blue, 0 }  ,fill opacity=0 ]   (366.57,416.5) -- (529.11,339.76) ;
\draw [color={rgb, 255:red, 0; green, 0; blue, 0 }  ,draw opacity=1 ][fill={rgb, 255:red, 0; green, 0; blue, 0 }  ,fill opacity=0 ]   (447.84,122.34) -- (447.84,378.13) ;
\draw [color={rgb, 255:red, 0; green, 0; blue, 0 }  ,draw opacity=1 ][fill={rgb, 255:red, 0; green, 0; blue, 0 }  ,fill opacity=0 ]   (197.77,122.34) -- (447.84,122.34) ;
\draw [color={rgb, 255:red, 0; green, 0; blue, 0 }  ,draw opacity=1 ][fill={rgb, 255:red, 0; green, 0; blue, 0 }  ,fill opacity=0 ] [dash pattern={on 0.84pt off 2.51pt}]  (241.53,416.5) -- (404.08,339.76) ;
\draw [color={rgb, 255:red, 0; green, 0; blue, 0 }  ,draw opacity=1 ][fill={rgb, 255:red, 0; green, 0; blue, 0 }  ,fill opacity=0 ] [dash pattern={on 0.84pt off 2.51pt}]  (329.06,122.34) -- (329.06,378.13) ;
\draw [color={rgb, 255:red, 0; green, 0; blue, 0 }  ,draw opacity=1 ][fill={rgb, 255:red, 0; green, 0; blue, 0 }  ,fill opacity=0 ]   (116.5,288.61) -- (366.57,288.61) ;
\draw [color={rgb, 255:red, 0; green, 0; blue, 0 }  ,draw opacity=1 ][fill={rgb, 255:red, 0; green, 0; blue, 0 }  ,fill opacity=0 ] [dash pattern={on 0.84pt off 2.51pt}]  (279.04,211.87) -- (529.11,211.87) ;
\draw [color={rgb, 255:red, 0; green, 0; blue, 0 }  ,draw opacity=1 ][fill={rgb, 255:red, 0; green, 0; blue, 0 }  ,fill opacity=0 ] [dash pattern={on 0.84pt off 2.51pt}]  (116.5,288.61) -- (279.04,211.87) ;
\draw [color={rgb, 255:red, 0; green, 0; blue, 0 }  ,draw opacity=1 ][fill={rgb, 255:red, 0; green, 0; blue, 0 }  ,fill opacity=0 ]   (529.11,83.98) -- (529.11,339.76) ;

\draw  [draw opacity=0][fill={rgb, 255:red, 234; green, 230; blue, 115 }  ,fill opacity=0.35 ] (529.63,211.89) -- (529.63,288.62) -- (529.63,339.78) -- (448.36,378.15) -- (329.58,378.15) -- (329.58,250.26) -- (404.08,211.13) -- cycle ;
\draw  [draw opacity=0][fill={rgb, 255:red, 130; green, 150; blue, 246 }  ,fill opacity=0.51 ] (322.8,122.34) -- (322.8,250.24) -- (322.8,250.24) -- (255.92,281.34) -- (240.28,288.61) -- (115.25,288.61) -- (115.25,160.71) -- (196.52,122.34) -- (215.9,122.34) -- (239.03,122.34) -- cycle ;
\draw [line width=2.25]    (116.26,185.86) -- (116.74,399.36) ;
\draw [shift={(116.75,403.36)}, rotate = 269.87] [color={rgb, 255:red, 0; green, 0; blue, 0 }  ][line width=2.25]    (17.49,-5.26) .. controls (11.12,-2.23) and (5.29,-0.48) .. (0,0) .. controls (5.29,0.48) and (11.12,2.23) .. (17.49,5.26)   ;
\draw [shift={(116.26,185.86)}, rotate = 269.87] [color={rgb, 255:red, 0; green, 0; blue, 0 }  ][line width=2.25]    (17.49,-7.84) .. controls (11.12,-3.68) and (5.29,-1.07) .. (0,0) .. controls (5.29,1.07) and (11.12,3.68) .. (17.49,7.84)   ;
\draw [line width=2.25]    (326.27,149.86) -- (326.73,229.36) ;
\draw [shift={(326.75,233.36)}, rotate = 269.67] [color={rgb, 255:red, 0; green, 0; blue, 0 }  ][line width=2.25]    (17.49,-5.26) .. controls (11.12,-2.23) and (5.29,-0.48) .. (0,0) .. controls (5.29,0.48) and (11.12,2.23) .. (17.49,5.26)   ;
\draw [shift={(326.27,149.86)}, rotate = 269.67] [color={rgb, 255:red, 0; green, 0; blue, 0 }  ][line width=2.25]    (17.49,-7.84) .. controls (11.12,-3.68) and (5.29,-1.07) .. (0,0) .. controls (5.29,1.07) and (11.12,3.68) .. (17.49,7.84)   ;
\draw [line width=2.25]    (329.27,286.86) -- (329.73,365.36) ;
\draw [shift={(329.75,369.36)}, rotate = 269.67] [color={rgb, 255:red, 0; green, 0; blue, 0 }  ][line width=2.25]    (17.49,-5.26) .. controls (11.12,-2.23) and (5.29,-0.48) .. (0,0) .. controls (5.29,0.48) and (11.12,2.23) .. (17.49,5.26)   ;
\draw [shift={(329.27,286.86)}, rotate = 269.67] [color={rgb, 255:red, 0; green, 0; blue, 0 }  ][line width=2.25]    (17.49,-7.84) .. controls (11.12,-3.68) and (5.29,-1.07) .. (0,0) .. controls (5.29,1.07) and (11.12,3.68) .. (17.49,7.84)   ;

\draw (89,126) node [anchor=north west][inner sep=0.75pt]  [font=\LARGE] [align=left] {\textit{A{\scriptsize 000}}};
\draw (90,412) node [anchor=north west][inner sep=0.75pt]  [font=\LARGE] [align=left] {\textit{A{\scriptsize 020}}};
\draw (298,96) node [anchor=north west][inner sep=0.75pt]  [font=\LARGE] [align=left] {\textit{A{\scriptsize 101}}};
\draw (305,233) node [anchor=north west][inner sep=0.75pt]  [font=\LARGE] [align=left] {\textit{A{\scriptsize 111}}};
\draw (309,370) node [anchor=north west][inner sep=0.75pt]  [font=\LARGE] [align=left] {\textit{A{\scriptsize 121}}};
\draw (212,417) node [anchor=north west][inner sep=0.75pt]  [font=\LARGE] [align=left] {\textit{A{\scriptsize 120}}};
\draw (224,132) node [anchor=north west][inner sep=0.75pt]  [font=\LARGE] [align=left] {\textit{A{\scriptsize 100}}};

\end{tikzpicture}

\end{adjustbox}

	 \begin{align*} 
	 \left[\begin{tikzcd}[scale cd=0.6, row sep=scriptsize,column sep=scriptsize]
    	A_{000} \arrow[d, tail] \\
    	A_{020}\end{tikzcd}\right] &=  \left[\begin{tikzcd}[scale cd=0.6, row sep=scriptsize,column sep=scriptsize]
    	A_{101} \arrow[d, tail] \\
    	A_{111}\end{tikzcd}\right] +  \left[\begin{tikzcd}[scale cd=0.6, row sep=scriptsize,column sep=scriptsize]
    	A_{111} \arrow[d, tail] \\
    	A_{121}\end{tikzcd}\right],  \\
	\end{align*}

	which we can write as
	\begin{equation}\label{relation}
	[A_{020}] - [A_{000}] = [A_{121}] - [A_{101}].
	\end{equation}
 By Equation \eqref{eqn:mfdbundle}, we can represent these manifold bundles by a single manifold, for example their fibre over $\{ 0 \}$.

	Take the lower part of the 2-simplex to be given by the equalities in the vertical direction. Then we also have
\begin{equation}\label{relation2}
	[A_{010}] - [A_{000}] = [A_{111}] - [A_{101}]=[A_{011}] - [A_{001}],
	\end{equation}
	where the last equation follows from \eqref{square relation}.
	Finally consider the cube with
 $A_{000}= A_{001}=\varnothing$
\[
\left[\begin{tikzcd}[scale cd=0.6, row sep=scriptsize,column sep=scriptsize] 
&& \varnothing \arrow[dd, tail] \\
\varnothing \arrow[urr, squiggly, "\varnothing"] \arrow[dd, tail] &&\\
&&A_{011}\\
 A_{010} \arrow[urr, squiggly, "W"] &&
\end{tikzcd} \right].
\]

We deduce that 
 \begin{equation}\label{relation3}
	[A_{010}] - [\varnothing] = [A_{011}] - [\varnothing] ,
	\end{equation}

and hence we obtain a relation that cobordant manifolds 
 $A_{010}$ and $A_{011}$ represent the same element in $\pi_1 Y.$

The square relation in Equation \eqref{square relation} was shown in  \cite{SKpaper} to induce the $\SK^\d$ relation.

	At this point we have shown that cobordism and cut and paste relations are relations in $\pi_1 Y$. To show that these are the only relations, we need to check that the generic relation in Equation \eqref{relation} can be obtained using the cobordism and cut and paste relations.

  First note that using the cut and paste relation in a generic cube allows us to write
 $$[A_{020}] - [A_{000}]  = [A_{120}] - [A_{100}].$$
 From the cobordism relation, we know that 
 $$[A_{120}] - [A_{100}] = [A_{121}]- [A_{101}].$$
Combining these two equations we obtain the generic relation in Equation \eqref{relation}.

\end{proof}

\section{$K^\cube$ is a spectrum} \label{GammaSec}
We use Segal's machinery of $\Gamma$-spaces to show that the geometric realisation of the simplicial space $X_\bullet$ is an infinite loop space and hence gives an $\Omega$-spectrum. We first briefly recollect the basic definitions of $\Gamma$-spaces. 

Let $\Gamma^{\op}$ be a skeleton of the category of finite pointed sets. By this we mean that objects in $\Gamma^{\op}$ are given by the pointed sets $n_{+}=\{ *, 1, \ldots, n\}$ for non-negative integers $n$ and morphisms are pointed maps of sets. The object $0_{+}=\{ *\}$ is initial and terminal in $\Gamma^{\op}$. A \textit{$\Gamma$-space} is a covariant functor from $\Gamma^{\op}$ to the the category of pointed spaces $S_{*}$.
A functor from $\Gamma^{\op}$ to the the category $S^{\Delta^{\op}}$ of simplicial spaces is called a \textit{simplicial $\Gamma$-space}.

The $i$-th Segal map is a map $e_i \colon n_{+} \to 1_{+}$ defined by
 \[
 e_i(j)= 
 \begin{cases}
 *, ~j \neq i\\
 1, ~j=i.
 \end{cases}
 \]
 A $\Gamma$-space $A \colon \Gamma^{\op} \to S_{*}$ is called \textit{special} if the Segal map
  \[
  \prod_{i=1}^{n} A(e_i) \colon A(n_{+}) \longrightarrow \prod_{n} A(1_{+})
  \]
is a weak homotopy equivalence for all positive integers $n$. In this case the set $\pi_0(A(1_{+}))$ inherits the structure  of an abelian monoid with addition coming from the zig-zag
 $$\begin{tikzcd}
   A(1_{+}) \times A(1_{+}) &&  A(2_{+}) \arrow{ll}{\simeq}[swap]{A(e_1) \times A(e_2)} \arrow{r}{A(\triangledown)} & A(1_{+}),
 \end{tikzcd}$$
where the first map is the Segal map and the second map is induced by $\triangledown \colon 2_{+} \to 1_{+}$ with $\triangledown(1)=\triangledown(2)=1.$ $A$ is called \textit{very special} if it is special and the monoid $\pi_0(A(1_{+}))$ is a group.

Segal showed how to associate to a very special $\Gamma$-space $A$ an $\Omega$-spectrum whose infinite loop space is $A(1_{+})$ (see \cite{segal1974categories} and \cite{BF}). 

We now construct a $\Gamma$-space $A$ such that $A(1_{+})$ is the 
geometric realisation of $X_\bullet$. 
Let us set up some terminology. We say that an element $\alpha$ in $X_k$ is \textit{decomposable} if there exist non-zero elements $\beta, \gamma  \in X_k$ such that $\alpha=\beta \sqcup \gamma,$ where by a zero element we mean the degenerate $k$-simplex induced from the zero simplex given by the empty manifold $\varnothing \in X_0$. Recall that any element in $X_k$ is given by a $k \times k \times k$ cube diagram  over $\Delta^k$ - by the disjoint union of two such cubes we mean the cube given by the vertex-, edge-, and face-wise disjoint union (if it exists, i.e. if the corresponding embeddings into $\R^{\infty}$ are disjoint).
An example of a decomposable element in $X_1$ is depicted below (in this diagram edges  and faces of the cube should also be marked by the corresponding disjoint unions, but we have omitted these).

\begin{tikzcd}[scale cd=0.6, row sep=small,column sep=small]& 
    	A_0' \sqcup A_1'\arrow[rr, hook]\arrow[dd, tail] & & B_0' \sqcup B_{1}'\arrow[dd, tail] \\
    A_0 \sqcup	A_1 \arrow[rr,crossing over, hook] \arrow[ur, squiggly] \arrow[dd, tail] & & B_0 \sqcup B_1 \arrow[ur, squiggly]\\
    	& C_0' \sqcup C_{1}'\arrow[rr, hook] &  & D_0' \sqcup D_{1}' \\
    	C_0 \sqcup C_{1}\arrow[rr, hook] \arrow[ur, squiggly] & & D_0 \sqcup D_{1}\arrow[from=uu,crossing over, tail] \arrow[ur, squiggly]\\
    	\end{tikzcd}
    	\hspace{0.01cm}
    	$=$
    	\hspace{0.01cm}
\begin{tikzcd}[scale cd=0.6, row sep=small,column sep=small]& 
    	A_0'\arrow[rr, hook]\arrow[dd, tail] & & B_{0}'\arrow[dd, tail] \\
    	A_0\arrow[rr,crossing over, hook] \arrow[ur, squiggly] \arrow[dd, tail] & &  B_0 \arrow[ur, squiggly]\\
    	& C_{0}'\arrow[rr, hook] &  & D_{0}' \\
    	C_{0}\arrow[rr, hook] \arrow[ur, squiggly] & & D_{0}\arrow[from=uu,crossing over, tail] \arrow[ur, squiggly]\\
    	\end{tikzcd}
    	\hspace{0.01cm}
    	$\sqcup$
    	\hspace{0.01cm}
    	\begin{tikzcd}[scale cd=0.6, row sep=small,column sep=small]& 
    	A_1'\arrow[rr, hook]\arrow[dd, tail] & & B_{1}'\arrow[dd, tail] \\
    	A_1\arrow[rr,crossing over, hook] \arrow[ur, squiggly] \arrow[dd, tail] & &  B_1 \arrow[ur, squiggly]\\
    	& C_{1}'\arrow[rr, hook] &  & D_{1}' \\
    	C_{1}\arrow[rr, hook] \arrow[ur, squiggly] & & D_{1}\arrow[from=uu,crossing over, tail] \arrow[ur, squiggly]\\
    	\end{tikzcd}

Since we require all manifolds to be compact, every vertex of a cube diagram has only finitely many connected components. Hence every element in $X_k$ can be written as a disjoint union of finitely many indecomposable elements. It is not difficult to see that such a decomposition is unique up to permutation
and we will call the indecomposible summands of an element its \textit{indecomposable components}.

We define $\tilde{A}(n_{+})_{\bullet}$ to be a simplicial  set for which the  set of $k$-simplices $\tilde{A}(n_{+})_k$ given by pairs $(\alpha, L)$, where  $\alpha \neq \varnothing \in X_k$ and $L$ is a labelling of indecomposable components of $\alpha$ by non-basepoint elements of $n_{+}.$ For $n=0$ the simplicial  set $\tilde{A}(0_{+})_{\bullet}$ consists of one point being the empty manifold $\varnothing$. We treat the $\varnothing \in \tilde{A}(n_{+})_0$ (and the corresponding degenarate simplices) differently than others and don't equip it with any labelling.  
Note that $\tilde{A}(1_{+})_{\bullet}$ is just $X_\bullet$.

\begin{lemmasection}
The simplicial  sets $\{ \tilde{A}(n_{+})_{\bullet }\}_{n \geq 0}$ assemble into a simplicial $\Gamma$- set.
\end{lemmasection}
\begin{proof}
Given a morphism $f \colon m_{+} \to n_{+}$ of based sets we define the induced map 
\[
\tilde{A}(m_{+})_k \to \tilde{A}(n_{+})_k
\]
\[
(\alpha, L) \mapsto (\alpha', L')
\]
in the following way. We first relabel the indecomposable components of $\alpha$, i.e., the components that had label $i \in \{1, \ldots , m\}$ will be labelled by $f(i)$ instead. Next we remove those components that got labelled by the basepoint and call the resulting element $(\alpha', L')$. 
\end{proof}

\begin{lemmasection}
The simplicial $\Gamma$-set $\tilde{A}$ is levelwise special in the sense that the $\Gamma$-set $n_{+} \mapsto \tilde{A}(n_{+})_k$ is special for every $k$.
\end{lemmasection}
\begin{proof}
We need to show that for every $k$ the Segal map is a weak homotopy equivalence
\begin{align*}
\tilde{A}(n_{+})_k & \xrightarrow{\prod_{i=1}^{n} \tilde{A}(e_i)} \displaystyle 
   \prod_{n} \tilde{A}(1_{+})_k\\
   (\alpha, L) & \longmapsto  (\alpha_1, \ldots, \alpha_n),
\end{align*}
where $\tilde{A}(e_i)$ sends $(\alpha, L)$ to $\alpha_i$, the union of indecomposable components of $\alpha$ that are labelled by $i$. Note that in $\tilde{A}(1_{+})$ there is no choice for the labelling, all components have to be labelled by the element $1$.
The image of the Segal map is given by the subset
\[
B\coloneqq \{(\beta_1, \ldots, \beta_n) \in \displaystyle \prod_{n} \tilde{A}(1_{+})_k ~|~ \beta_i \text{~disjoint from } \beta_j \text{ for } i \neq j\},
\]
where two $k \times k \times k$ diagrams $\beta$ and $\gamma$ are called disjoint if the cobordisms (viewed as subsets in $[0,1] \times \R_+\times\R^\infty$) coming from the corresponding edges of the diagrams $\beta$ and $\gamma$ are disjoint. More precisely, if we denote by $\beta_{a,b,[c, c+1]}$ the restriction of the diagram $\beta$ to the edge that connects the vertices $(a,b,c)$ and $(a,b,c+1)$ and encodes an embedded cobordism, then we require the cobordisms $\beta_{a,b,[c, c+1]}$ and $\gamma_{a,b,[c, c+1]}$ to be disjoint for all $a,b \in \{0 , \ldots , k\}$ and $c \in \{0 , \ldots , k-1 \}$. 

We show that $B$ is a weak deformation retract of the product  set $\prod_{n} \tilde{A}(1_{+})_k$. Analogous to the way it was done in \cite{nguyen2017infinite} for the usual cobordism category, we use the ``shift-map'' to make the diagrams disjoint. For the sake of completeness we present the proof here, but we recommend to check \cite{nguyen2017infinite} for a beautiful visualisation of the shift-map. 

For $r \in \R$ consider the \textit{shift-map}
 \[
 F_r \colon \R^{\infty} \to \R^{\infty}
 \]
 \[
 (x_1, x_2, \ldots) \mapsto (r, x_1, x_2, \ldots).
 \]
It is clear that $F_r$ is homotopic to the identity map.  For any non-negative integer $k$ the map $F_r$ induces a self-map of $X_k=\tilde{A}(1_{+})_k$ by taking the $k \times k \times k$ diagram  over $\Delta^k$ $\alpha \in X_k$ and postcomposing all of its vertices, edges, and faces (all of these are represented by the embeddings of manifolds/cobordisms) with $F_r$. Namely, given an embedding of a cobordism $W$ into $[0,1] \times \R_+\times\R^\infty$ we modify it by postcomposing with
\[
\id_{[0,1]} \times \id_{\R_+} \times F_r \colon [0,1] \times \R_+\times\R^\infty \to [0,1] \times \R_+\times\R^\infty.
\]
Analogously the postcomposition with $F_r$ is defined for the embedded manifolds, SK-embeddings of manifolds, and SK-embeddings of cobordisms. These actions agree and assemble into a new diagram $F_r \circ \alpha \in X_k.$ The induced map $X_k \to X_k$ is again homotopic to the identity.

Choose $n$ pairwise different real numbers $r_1, \ldots, r_n$ and consider the map 
\[
\prod_{n} \tilde{A}(1_{+})_k \to \prod_{n} \tilde{A}(1_{+})_k
\]
given by applying $F_{r_i}$ to the $i$-th factor. This map is the desired deformation retraction.
\end{proof}
 
 \begin{cor}
The $\Gamma$-space $A\colon \Gamma^{\op} \to S_{*}$ defined by realising $\tilde{A}$ at every $n$, i.e. $$A(n_{+}) \coloneqq ||\tilde{A}(n_{+})_{\bullet}||,$$ is special.

 \end{cor}
 \begin{proof}
 Simplicial map that is degreewise weak homotopy equivalence induces weak homotopy equivalence on fat realisations, hence
  $$\begin{tikzcd}
   ~|| \tilde{A}(n_{+})_{\bullet}|| \arrow{rr}{\prod_{i=1}^{n} \tilde{A}(e_i)}[swap]{\simeq} &&  \displaystyle \prod_{n} ||\tilde{A}(1_{+})_{\bullet}||. 
\end{tikzcd}$$ 
 \end{proof} 
Finally, the space $||\tilde{A}(1_{+})_{\bullet}||$ is connected and therefore  the $\Gamma$-space $A$ is very special and we conclude that $||X_\bullet||$ is an infinite loop space as was claimed.

\section{Maps from $B \overline{\Cob}_{n+1}^{\partial_{triv}}$ and $K^{\square}(\mbar)$} \label{induced maps section}

In this section we consider maps into the constructed spectrum $K^{\cube}(\mbar)$.

The simplicial  set $X^*_{0,0,\bullet}$ is the nerve of a unitalisation (in the sense of \cite{ebert2019semisimplicial}) of a  parameterised version (see \cite{raptis2017parametrized}) of the cobordism category with trivial boundary $\Cob_{n+1}^{\partial_{triv}}$ described in Section \ref{section:fibrations}. Note that the parametrised cobordism category is defined as a simplicial object in non-unital categories in \cite{raptis2017parametrized}, and we here use a unital version.
We write $(N_\bullet \Cob_{n+1}^{\partial_{triv}})_*$ for the bisimplicial set that is the parametrised nerve of the unital trivial boundary cobordism category.
\footnote{Proposition 3.8 in \cite{ebert2019semisimplicial} shows that unitalisation induces weak homotopy equivalences on classifying spaces for weakly unital topological categories, although this is a slightly different setting. }

Let $i$ be the inclusion map  
$$i: (N_\bullet \Cob_{n+1}^{\partial_{triv}})_* \rightarrow X^*_{\bullet, \bullet, \bullet}.$$
Note that on $\pi_0$ of realisations this map is trivial since  $|X^*_{\bullet, \bullet, \bullet}|$ is connected.  Denote by $B\C_{n+1}$ the realisation $|(N_\bullet \Cob_{n+1}^{\partial_{triv}})_*|$.

\begin{thm}
For

$\Omega|i| \colon \Omega B\C_{n+1} \rightarrow  K^\cube (\mbar)$ the map induced by the inclusion, $$
\pi_1 |i|: \pi_1 B\C_{n+1} \rightarrow \pi_1 | X^*_{\bullet, \bullet, \bullet}| \cong \SKbar_{n}^\partial$$ is the 0 map.
\end{thm}
\begin{proof}
 $ B\C_{n+1}$ is the realisation of a bisimplicial set, hence $\pi_1 B\C_{n+1}$ is generated by 1-simplices in the diagonal modulo relations imposed by 2-simplices in the diagonal.
Under the map $i$, any 1-simplex in $(N_\bullet \Cob_{n+1}^{\partial_{triv}})_*$ is mapped to a cube of the shape $\left[\begin{tikzcd}[scale cd=0.6, row sep=scriptsize,column sep=scriptsize] & B' \\
    	B \arrow[ur, squiggly, "W"] & \end{tikzcd} \right]$ 
in $X^*_{\bullet, \bullet, \bullet}$, which by the proof Theorem \ref{thm:pi_0} is zero in $ | X^*_{\bullet, \bullet, \bullet}|.$ Hence, $\pi_1 |i|$ is the zero map.


\end{proof}

  Another (tri)simplicial subset of $X^*_{\bullet, \bullet, \bullet}$ that is of interest to us is $X^*_{\bullet, \bullet, 0}$, defined in Section \ref{X_110}
  A spectrum associated to the infinite loop space $\Omega | X^*_{\bullet, \bullet, 0}|$ is the topological version of the cut and paste spectrum $K^{\square}(\Mnfldbd_n)$ defined in \cite{SKpaper} and will be denoted by $K^{\square}(\mbar)$. Here the bar denotes that we are taking topology on manifolds into account by means of the extra simplicial direction.
  
Recall that $K_0^{\square}(\Mnfldbd_n) \cong \SK_{n}^\d.$ We show that the topological cut and paste spectrum also has $\SK_{n}^\d$ as its zeroth homotopy group.
 \begin{thm}\label{thm:normalSK}
 There is an isomorphism
 \[
 K_0^{\square}(\mbar) \cong \SK_{n}^\d.
 \]
 \end{thm}
 \begin{proof}
We want to compute  $\pi_1 | X^*_{\bullet, \bullet, 0}|$. 
 Write $$Z = |\text{diag}(X^*_{\bullet, \bullet, 0})|.$$
The 2-simplices of the form given by Equation (\ref{eqn:contracting}) are all in $Z$ hence we can perform the contraction of vertices as before.
Now Equations (\ref{eqn:firstone})-(\ref{square relation}) in the proof of Theorem \ref{thm:pi_0} show that the $\pi_1 Z$ satisfies the SK relations.
To conclude that these are all the relations in $\pi_1 Z$ one can show that the relations coming from a generic 2-simplex given by a $2 \times 2$ square  over $\Delta^2$ can be expressed in terms of the SK relations.
 \end{proof}
 We denote the map of spectra induced from the inclusion  $X^*_{\bullet, \bullet, 0} \to X^*_{\bullet, \bullet, \bullet}$ by
 \[
 j \colon K^{\square}(\mbar) \to K^\cube(\mbar).
 \]
 \begin{thm}
  Under the identifications $K_0^{\square}(\mbar) \cong \SK_{n}^\d$ and $K_0^{\cube}(\mbar) \cong \SKbar_{n}^\d$  $j_0$ induces the quotient map 
  \[
  \SK_{n}^\d \rightarrow \SKbar_{n}^\d.
  \]
 \end{thm}
 \begin{proof}
 This is evident from comparing the proofs of Theorem \ref{thm:pi_0} and Theorem \ref{thm:normalSK}.
 \end{proof}
 
\begin{remarksection}\upshape
Analogously to how it was done in Section \ref{GammaSec}, one can show that the spaces $K^{\square}(\mbar)$ and $\Omega B\C_{n+1}$ are infinite loop spaces. As before the $\Gamma$-space structure comes from the operation of taking the disjoint union of manifolds. One can also check that the obvious map between the corresponding $\Gamma$-spaces for $\Omega B\C_{n+1}$,
$K^{\square}(\mbar)$, and $K^\cube (\mbar)$ restricts to the maps $i$ and $j$ on the first spaces. Therefore, $\Omega |i|$ and $j$ are maps of spectra.
\end{remarksection}

 \bibliographystyle{amsalpha}
  \bibliography{biblio}

\end{document}